\documentclass[leqno,12pt]{amsart} 
\usepackage{amssymb}
\usepackage{tikz,tikz-cd}
\usetikzlibrary{matrix,arrows.meta}
\theoremstyle{plain} 
\newtheorem{theorem}{\indent\sc Theorem}[section]
\newtheorem{lemma}[theorem]{\indent\sc Lemma}

\newtheorem{proposition}[theorem]{\indent\sc Proposition}

\theoremstyle{definition} 

\newtheorem{remark}[theorem]{\indent\sc Remark}
\newtheorem{example}[theorem]{\indent\sc Example}

\usepackage[active]{srcltx}
\usepackage{pdfsync}
\usepackage{tabularx} 
\usepackage{arydshln}

\newcommand{\CC}{\mathbb{C}}

\newcommand{\ZZ}{\mathbb{Z}}
\newcommand{\HH}{\mathbb{H}}
\newcommand{\OO}{\mathbb{O}}
\newcommand{\RR}{\mathbb{R}}

\newcommand{\SO}{\mathrm{SO}}
\newcommand{\SU}{\mathrm{SU}}
\newcommand{\U}{\mathrm{U}}
\newcommand{\jj}{\mathbf{j}}
\newcommand{\ii}{\mathbf{i}}
\newcommand{\kk}{\mathbf{k}}
\renewcommand{\ll}{\mathbf{l}}
\renewcommand{\hom}{\mathrm{Hom}}

\def\dim{\mathop{\hbox{\rm dim}}}

\newcommand{\der}{\mathfrak{der}}

\newcommand{\id}{\mathrm{id}}
\newcommand{\mm}{\mathfrak{m}}
\newcommand{\hh}{\mathfrak{h}}

\def\span#1{\langle #1\rangle}
 
\title[  Homogeneous spaces of $G_2$]{  Reductive homogeneous spaces \\of the compact Lie group $G_2$}

\author[C.~Draper, F.J.~Palomo]{ 
Cristina Draper${}^*$ and Francisco J. Palomo{${}^\star$}  
\bigskip \\ 
}  

\thanks{${}^*$ Supported by    Junta de Andaluc\'{\i}a  through projects  FQM-336, UMA18-FEDERJA-119, and PAIDI project P20\_01391, and  by the Spanish Ministerio de Ciencia e Innovaci\'on   through projects  PID2019-104236GB-I00/AEI/10.13039/501100011033    and PID2020-118452GB-I00, all of them with FEDER funds.}
\thanks{${}^\star$    Partially
supported by Spanish MICINN project PID2020-118452GB-I00 and Andalusian and ERDF projects FQM-494 and  P20$_{-}$01391.}

\address{%
Departamento de \'Algebra, Geometr\'\i a y Topolog\'\i a,\endgraf
Universidad de M\'{a}laga, 
 29071 M\'{a}laga,  
Spain
}
\email{cdf@uma.es}
\address{%
Departamento de Matem\'atica Aplicada,\endgraf
Universidad de M\'{a}laga, 
 29071 M\'{a}laga,  
Spain
}
\email{fpalomo@uma.es}
\begin{document}


\maketitle  

\vspace{-45pt}\begin{center}\scriptsize{\textrm{Dedicated to Alberto Elduque on his 60th birthday}}\end{center}

\begin{abstract}
The first author defended her doctoral thesis \lq\lq Espacios homog\'eneos reductivos y \'algebras no asociativas\rq\rq\  in 2001, supervised by P. Benito and A. Elduque. This thesis contained the classification of the Lie-Yamaguti algebras with standard enveloping algebra $\mathfrak{g}_2$ over fields of characteristic zero, which in particular gives the classification of the homogeneous reductive spaces of the compact Lie group $G_2$. In this work we   revisit this classification from a more geometrical approach. We provide too geometric models of the corresponding homogeneous spaces and   make explicit some relations among them. 
 \end{abstract}

\section{Introduction }

This paper extends the talk \emph{In the footsteps of Alberto Elduque} given by the first author in the  conference \lq\lq Non-associative algebra and related topics\rq\rq\, (Coimbra, July 2022), dedicated to honor Alberto Elduque on the occasion of his 60th birthday. We thought that seeing his student following his steps and developing his ideas along the years would be a nice way of thanking him for so many years of friendship, for sharing ideas and adventures. And, of course, for sharing our love of exceptional Lie algebras and exceptional objects.   \smallskip

The geometries associated to exceptional Lie groups often 
present interesting properties which permit both to understand these exceptional algebraic objects and to shed some light on the geometric features of homogeneous manifolds for such groups.  The smallest   exceptional Lie groups have real dimension $14$, those of type $\textrm{G}_2$, namely, the automorphism groups of the octonion division algebra $\OO$ and of   the split   division algebra $\OO_s$, and the double covering of the last one.
We will  focus on the connected and simply-connected compact Lie group $G_2=\mathrm{Aut}(\OO)$.

For the first author, the Lie group $G_2$  has been a traveling companion throughout the years. The first of these encounters occurred   in her doctoral dissertation  \cite{tesis}, where Lie-Yamaguti algebras with standard enveloping algebra of type $\textrm{G}_2$ over fields of characteristic zero were classified. 
Lie-Yamaguti algebras are binary-ternary algebras defined to codify reductive homogeneous spaces to an algebraic structure, similarly to the way Lie triple systems are defined to codify symmetric spaces.
In particular, the above mentioned classification  essentially gives the list of the homogeneous reductive spaces of the   Lie groups of type $\textrm{G}_2$. 
Motivated by these facts and by the very interesting geometrical structures associated with $G_2$, 
our main purpose will be to give geometric descriptions of the homogeneous reductive spaces of the compact Lie group $G_2$.
When we began to look for descriptions of these homogeneous manifolds, we found that some of them were well-known. However, not only these results were scattered in the literature, but also we found insufficient information flow between differential geometers and researchers in Algebra. For instance, until now the cites to \cite{LYg2} were from researchers working in topics related to nonassociative   structures. Also the first reference we have found about the quotient $G_2/\SO(3)$ is \cite{Enoyoshi}, only two years ago. We have tried to make accessible the results of  \cite{LYg2}, while giving a unified perspective. For instance, the homogeneous spaces $G_2/Sp(1)^+$ and $G_2/Sp(1)^-$ are   considered in   the book   \emph{Sasakian Geometry} \cite[Example~13.6.8]{libroBG}, which offers a very nice  panoramic of the diagram submersions, but we add not only a self-contained description, but also the relations with the remaining $G_2$-homogeneous spaces.

The exceptional Lie group $G_2$ occurs in different situations and in various guises in Differential Geometry (see \cite{AgricolaG2} for an historical perspective). With no claim to be exhaustive, we would like to recall several such contexts. 
$G_2$ may be the holonomy group of certain non locally symmetric $7$-dimensional Riemannian manifolds,  according to the Berger list of irreducible  holonomies (see for instance \cite[Chap. 10]{Besse}). Note also that $G_{2}$ is a subgroup of $\SO(7)$ and then, one can also consider $G_2$-structures on $7$-dimensional Riemannian manifolds. Essentially,  a $G_2$-structure  consists on a $3$-form
which permits to construct a Riemannian metric,  a volume form and a  vector cross product \cite{G2estruc,Hitchin}.
Another issue is related to generic distributions in dimension $5$.
Let us recall that a rank two distribution $\mathcal{D}$ on a $5$-dimensional manifold $M$ is said to be generic if and only if it is bracket generating with grow vector $(2,3,5)$. That is, the Lie brackets of vector fields in $\mathcal{D}$ span a rank $3$-distribution of $TM$ and triple Lie brackets of vector fields in $\mathcal{D}$
span all $TM$.
The study of such distributions has a long history starting with the ``five variables paper'' by E. Cartan in 1910 \cite{Cartan2}. Rank two distributions arise in the mechanical system of a surface rolling without slipping and twisting on another surface. In this case, the configuration space has a rank two distribution which encodes the no slipping and twisting condition. When both surfaces are round spheres with ratio of their radii $1:3$, the universal double covering of the configuration space is a homogeneous space for the real split form of the exceptional Lie group of type $\mathrm{G}_{2}$.

As mentioned above, our main aim will be to describe the family of reductive $G_2$-homogeneous spaces in   geometrical terms, as well as  relationships between them. We have tried to do it in an accessible  way, while keeping a unified perspective. 
The paper is organized as follows. Section $2$ is devoted to generalities on the octonion division algebra $\OO$, and also to fix  some terminology and notations. For instance,  we  describe $G_2$   and its related Lie algebra $\mathfrak{g}_2$, not only in terms of automorphisms and derivations of the octonion algebra, but also in terms of a convenient (generic) 3-form.
The key result for this paper is Theorem~\ref{th_clasif}, which introduces the complete list of reductive Lie subalgebras of $\mathfrak{g}_{2}.$ We have labeled every such subalgebra as $\mathfrak{h}_{i}$ with $i\in \{1,\dots , 8\}$.  In order to construct the corresponding $G_2$-homogeneous manifold, we consider for every Lie subalgebra $\mathfrak{h}_{i}$ of $\mathfrak{g}_2$ the unique closed connected Lie subgroup $H_i$ with Lie algebra  $\mathfrak{h}_{i}$ by means of
Theorem~\ref{th_clasif} and \cite[Theorem~3.19]{Warner}.
All of these algebras  $\mathfrak{h}_{i}$ admit natural descriptions in terms of quaternions, complex numbers and derivations, with the exception of $\hh_8$, a principal three-dimensional subalgebra of $\mathfrak{g}_2$. The description of $\hh_8$ has to wait until Proposition~\ref{prop_TDSppal}, after devoting an effort  in Section~\ref{se_ppal} to understand the main properties and the existence of the principal subalgebras of compact real forms. In Theorem~\ref{th_clasif},   the corresponding 
   3-dimensional simple   subalgebra of $\mathfrak{g}_2$ is described by its properties, namely, the fact  that $\OO_0$ is an absolutely irreducible  module for it. Its   uniqueness  is clarified in  Proposition~\ref{pr_conj}.

The central part of this paper is Section~$3$, where we develop one by one every reductive $G_2$-homogeneous space. We have included old and new geometric descriptions of such spaces and several features and uses in Differential Geometry. 
Section $3$ begins with some generalities on homogeneous manifolds. Then, we realize $G_2$ as a hypersurface of  
the Stiefel manifold $V_{7,3}$ of all orthonormal $3$-frames in $\mathbb{R}^{7}$. The long string of reductive $G_2$-homogeneous manifolds starts with arguably the best known examples in the literature: the $8$-dimensional quaternion-K\"ahler symmetric space $G_{2}/\SO(4)$  and the nearly K\"ahler six dimensional sphere $\mathbb{S}^{6}\cong G_{2}/\SU(3).$ The   directed tree in Figure~\ref{tree} describes the relationships between the manifolds in Section $3$. According to the dimensions, the Lie group $G_2$ is in the top (the root) and $G_{2}/\SO(4)$ and the nearly K\"ahler sphere $\mathbb{S}^{6}\cong G_{2}/\SU(3)$ are in the bottom. Our directed tree has three leaves: $G_{2}/\SO(4)$, the nearly K\"ahler sphere $\mathbb{S}^{6}\cong G_{2}/\SU(3)$ and  the irreducible space $G_{2} /\SO(3)^{\textrm{irr}}$. Moreover, every arrow denotes a fiber bundle projection $G/H_{i}\to G/H_{j}$  with standard fiber $H_{j}/H_{i}$ whenever $H_{i}$ is a closed subgroup of $H_j$.

There are two branches with end on $\mathbb{S}^6$. The first one has two nodes: the unit tangent bundle $\mathcal{U}\mathbb{S}^{6}$ and the complex quadric $Q_5\subset \mathbb{C}P^6$ (excluding  the root, $G_2$).  Alternative descriptions and properties of the manifolds in this branch can be found in Sections~\ref{elM2}, \ref{elM3} and \ref{elM6}. The other branch with end on $\mathbb{S}^{6}$ has only the node $G_{2}/\SO(3)$. This $G_2$-homogeneous manifold $ G_2/\SO(3)$  is almost unknown, or at least   one  would think so from its virtually non-existent appearances in the literature.  Section~\ref{elM7} is devoted to this case.
Three branches end on the $8$-dimensional quaternion-K\"ahler symmetric space $G_{2}/\SO(4)$  (compare again with  \cite[Example~13.6.8]{libroBG}). Two branches agree with the case of $\mathbb{S}^{6}$. The third one has two nodes: $G_{2}/{\SU(2)^r}$ and $G_{2}/{\U(2)^r}$. Sections~\ref{elM4} and \ref{elM5} provide very concrete descriptions of such manifolds of twistor spaces over $G_{2}/\SO(4)$.
Observe that the first and the third branches with end $G_{2}/\SO(4)$ contain topologically different manifolds. In fact, the two copies of $\U(2)$ and $\SU(2)$ in $G_{2}$ produce quotients with different homotopy types \cite[Example~13.6.8]{libroBG} (see also \cite{Nakata}, which computes their third   homotopy groups   making use of the notion of Dynkin index).
Finally, the manifold $G_{2}/\SO(3)^{\textrm{irr}}$ is a leaf in our tree and its branch reduces only to one node, because the corresponding algebra is at the same time a maximal subalgebra, and minimal among the non abelian reductive ones. The material on this manifold, an isotropy irreducible space, is presented in Section~\ref{se_irr}, where we give the first concrete description (as far as we know) of this homogeneous space. Although Wolf gives in  \cite{Wolf68} a structure theory
and classification for non-symmetric, isotropy irreducible homogeneous spaces, the truth is that the mere apparition of $G_{2}/\SO(3)^{\textrm{irr}}$ in a list leaves the reader with more questions than answers.

 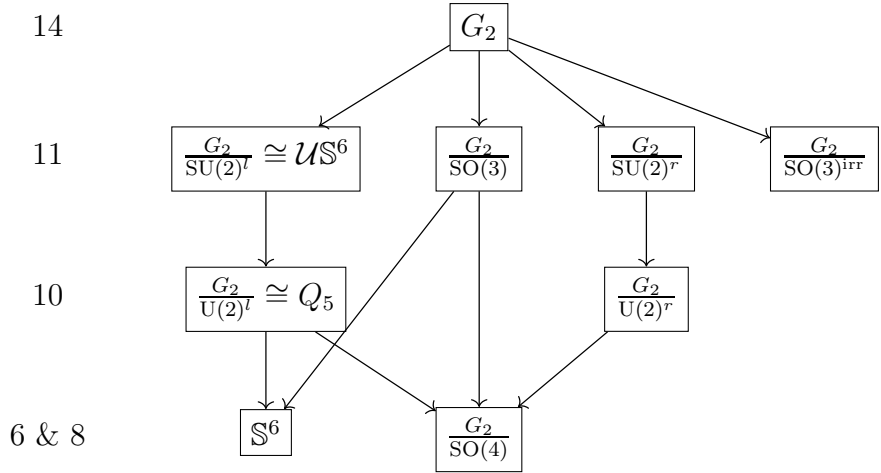
\begin{figure}[!h]
\centering
\begin{tikzpicture}     
  \matrix (m)  [
    matrix of nodes,
    nodes={draw},
    column sep=10mm,
    row sep=10mm,
  ] {
    |[draw=none]|  14   &          &    $G_2$       &\\
    |[draw=none]| 11   &   $\frac{G_2}{\SU(2)^l} \cong\mathcal{U} \mathbb S^6$       & $\frac{G_2}{\SO(3)}$    &  $\frac{G_2}{\SU(2)^r}$    &  $\frac{G_2}{\SO(3)^{\textrm{irr}}}$ \\
    |[draw=none]| 10   & $\frac{G_2}{\mathrm{U}(2)^l} \cong Q_5$  &       &$\frac{G_2}{\mathrm{U}(2)^r}$&  &  \\
    |[draw=none]| 6 $\&$ 8      & $\mathbb S^6$  & $\frac{G_2}{\SO(4)}$  & &   &   \\
  };
  \begin{scope}[
    font=\footnotesize,
    inner sep=.25em,
    every node/.style={fill=white},
  ]
    ;
    \path[commutative diagrams/.cd, every arrow] 
    (m-1-3)  edge  (m-2-2)
(m-1-3) edge  (m-2-3)
(m-1-3)  edge  (m-2-4)
(m-1-3)  edge  (m-2-5)
(m-2-2) edge (m-3-2)
(m-2-4) edge (m-3-4)
(m-3-2) edge (m-4-2) 
(m-3-2) edge (m-4-3)
(m-3-4) edge (m-4-3)
(m-2-3) edge (m-4-2)
(m-2-3) edge (m-4-3)
         ;
  \end{scope}
\end{tikzpicture}
\caption{Reductive homogeneous spaces of $G_2$}\label{tree}
\end{figure}


 \section{ Background on octonion algebras and $\mathfrak{g}_2$}
 
 In this paper we will focus on the real number field, due to the applications to Differential Geometry, although this section can be widely generalized.
 The main facts on Cayley algebras and its derivation algebras can be consulted in \cite{Schafer} and \cite{librorusos}.

 \subsection{ Octonion division algebra}\ 

 The well-known octonion division algebra is denoted here by  $\OO$ and a basis is given by $\{1, e_i:1\le i\le 7\}$, where the multiplication of such elements 
 follow the rules
$$
e_ie_{i+1}= 
e_{i+3},\qquad e_i^2=-1,
$$
for the indices modulo 7.
Moreover, if $ e_{ i}e_{j }=e_{k }$, then we also have $e_{j }e_{ k}=e_{ i}$ and $e_{k }e_{ i}=e_{ j}$, and $e_ie_j=-e_je_i$ for
$i\ne j$.  Alternatively, we can take $\OO=\HH\oplus\HH{\bf{l}}$, where $\HH=\langle 1,\bf{i},\bf{j},\bf{k}\rangle$ denotes the usual quaternion algebra 
with product, for any $q_1,q_2\in\HH$,
\begin{equation}\label{eq_CD}
q_1(q_2{\bf{l}})=(q_2q_1){\bf{l}},\qquad
(q_2{\bf{l}})q_1=(q_2\overline{q_1}){\bf{l}}\qquad \text{ and } \qquad 
(q_1{\bf{l}})(q_2{\bf{l}})=-\overline{q_2}q_1.
\end{equation}

Recall that $\OO$ is a quadratic algebra, that is,  every $x=a_01+\sum_{i=1}^7a_ie_i \in \OO$ satisfies the following second degree polynomial equation
$$
x^2-t(x)x+n(x)1=0,
$$
where $\bar x=a_01-\sum_{i=1}^7a_ie_i$, the trace is given by $t(x)=x+\bar x=2a_0$ and the norm by $n(x)=x\bar x= \sum_{i=0}^7{a_i}^2$.
 Moreover the norm is multiplicative, $n(xy)=n(x)n(y)$ for all $x,y\in\OO$, so that $\OO$ is a composition algebra. 
 We will also denote by $n\colon \OO\times\OO\to\RR$  the polar form $n(x,y)=\frac12\big(n(x+y)-n(x)-n(y)\big)$ related to the positive definite quadratic form $n$. Every $x\in \OO \setminus \{0\}$ has an inverse given by $x^{-1}=\bar x/n(x)$. 
 The octonion algebra is  an important example of nonassociative algebra: if we denote the associator by $(x,y,z):=(xy)z-x(yz)$, note that, for instance, $({\bf{i}},{\bf{j}},{\bf{l}})=2{\bf{k}}{\bf{l}}\ne0$.

\begin{remark} \label{re_subalg}
A remarkable property of the octonion algebra is that every subalgebra $\mathcal Q$ of $\OO$ of dimension 4 is isomorphic to $\HH$. Moreover, if we take  $v\in\mathcal Q^\perp$ with $n(v)=1$, then the isomorphism $f\colon \mathcal Q\to \HH$ is extended to an automorphism  $ \OO= \mathcal Q\oplus \mathcal Q v\to\OO=\HH\oplus\HH\ll$ by 
means of $q_1+q_2v\mapsto f(q_1)+f(q_2)\ll$. The proof is consequence of the fact that $v$ and $\mathcal Q$ satisfy the relations in Eq.~\eqref{eq_CD} (see for instance \cite[Chapter~2, Lemma~6]{librorusos}).
\end{remark}

\subsection{Cross products  and $3$-forms  }

The projection of the   product of the octonion algebra over the subspace of the zero trace elements $\OO_0=\{x\in \OO: t(x)=0\}$ defines a cross product on $\OO_0$ as follows
$$
\times\colon \OO_0\times \OO_0\to \OO_0,\qquad x\times y=\mathrm{pr}_{\OO_0}(xy)=xy-\frac12 t(xy)1.
$$
That is, a binary product satisfying $n(x\times y,x)=n(x\times y,y)=0$ and 
$$
n(x\times y)=\left| \begin{array}{cc} n(x,x)&n(x,y)\\n(y,x)&n(y,y)  \end{array}    \right|.
$$

 Equivalently we have a cross product in $\RR^7$ given by the natural identification as vector spaces $\OO_0\to\RR^7$,
$e_i\mapsto e_i$ (now $\{e_i\}_{i=1}^7$ denotes the canonical basis of $\RR^7$).
Moreover, the trilinear map $\Omega\colon\OO_0\times\OO_0\times\OO_0\to\RR$ defined by 
$$
\Omega(x,y,z)=n(x,y\times z)=n(x,yz),
$$
 is alternating and so defines a 3-form. It is frequently called the associative calibration on $\OO_0$ \cite{Harvey}.

\subsection{The automorphism group }

F.~Engel proved in \cite{Engel} that the compact Lie group $G_2$ is the isotropy
group of a generic 3-form in 7 dimensions (for instance, $\Omega$ is a  generic 3-form).
On the other hand, E.~Cartan proved that $G_2$ is also the automorphism group of the octonion algebra \cite{Cartan}. We use here both approaches: the classification of the reductive homogeneous spaces of $G_2$ in \cite{LYg2} follows the viewpoint of  the automorphism group of the octonion algebra, but we will profusely use  the  $3$-form $\Omega$   to provide concrete descriptions of such homogeneous spaces in Section~$3$.

So, we first think  of $G_2$   as the automorphism group 
$$\mathrm{Aut}(\OO)=\{f\in\mathrm{GL}(\OO):f(xy)=f(x)f(y) \ \forall x,y\in\OO\}.
$$ 
Since every automorphism preserves the norm, we have $\mathrm{Aut}(\OO)\subset \SO(\OO,n).$ 
Moreover, every automorphism $f$ satisfies $f(1)=1$, and taking into account that  $\OO_0= \langle1\rangle^\perp$, we get that $f$ can be restricted to $\OO_0$. This restriction determines the action of $f$ on $\OO$. Hence we can also consider $\mathrm{Aut}(\OO)$ as a subgroup of $\SO(\OO_0,n)$.

Recall that $\mathrm{GL}(\RR^7)\equiv \mathrm{GL}(\OO_0)$ acts on the set of alternating trilinear maps $\omega\colon
\OO_0\times \OO_0\times \OO_0\to\RR$
by $(f\cdot \omega)(x,y,z)=\omega(f^{-1}(x),f^{-1}(y),f^{-1}(z))$. The orbit of $\Omega$ is open and the group $G_2=\mathrm{Aut}(\OO)$ is isomorphic to the isotropy group $\{f\in\mathrm{GL}(\OO_0):f\cdot \Omega=\Omega\}$, by means of $f\mapsto f\vert_{\OO_0}$.

\subsection{The exceptional Lie algebra $\mathfrak{g}_2$  }

The 14-dimensional simple Lie algebra $\mathfrak{g}_2$  is the Lie algebra of derivations of the octonion algebra
$$
\der(\OO)=\{d\in\mathfrak{gl}(\OO):d(xy)=d(x)y+xd(y)\ \forall x,y\in\OO\},
$$
endowed with the usual commutator.
Similarly to the case of the group, the map $d\mapsto d\vert_{\OO_0}$ provides an isomorphism between $\der(\OO)$ and the Lie algebra
$$
  \{d\in\mathfrak{gl}(\OO_0):\Omega(d(x),y,z)+\Omega(x,d(y),z)+\Omega(x,y,d(z))=0\  \forall x,y,z\in\OO_0\}.
$$

In general, the derivations of an algebra are not easy to describe. In the case of the octonion algebra, the concrete computations are carefully developed in \cite[Chapter~8]{Schafer}. We will follow here this description. 
Let us denote by $L_x,R_x\colon \OO\to \OO$ the left and right multiplication operators given by $L_x(y)=xy$ and $R_x(y)=yx$. They are not derivations but behave well with respect to the norm, that is, if $x\in\OO_0$,
$$
L_x,R_x\in\mathfrak{so}(\OO,n)=\{f\in\mathfrak{gl}(\OO):n(f(x_1),x_2)+n(x_1,f(x_2))=0\ \forall x_i\in\OO\}.
$$
Now, let us denote by 
$$
D_{x,y}:=[L_x,L_y]+[L_x,R_y]+[R_x,R_y].
$$
Every $D_{x,y}$  is a derivation of $\OO$ such that $D_{x,y}(z)=[z,[x,y]]-3(x,z,y)$. Moreover, these operators span the whole derivation algebra, that is,
$$
\der(\OO)=\left\{\sum_{i=1}^kD_{x_i,y_i}:x_i,y_i\in\OO,k\in\mathbb N\right\}.
$$
Besides, the unique nonzero $\mathfrak{g}_2$-invariant map  $\OO_0\times \OO_0\to \der(\OO)$, up to scalar multiple, is given precisely by $(x,y)\mapsto D_{x,y}$. That is, for any $d\in \der(\OO)$, we have
\begin{equation}\label{eq_corcheteconds}
[d,D_{x,y}]=D_{d(x),y}+D_{x,d(y)}.
\end{equation}

\subsection{ Reductive subalgebras of $\mathfrak{g}_2$}

Recall that  a subalgebra $\hh$ of a Lie algebra $\mathfrak{g}$ is said to be \emph{reductive} if $\mathfrak{g}$ is completely reducible as $\hh$-module. 
 In particular there is  an $\hh$-submodule $\mm$ of $\mathfrak{g}$ such that $\mathfrak{g}=\hh\oplus\mm$, that is,
 $(\mathfrak{g},\hh)$ is a reductive pair.
 Take care because the converse is not necessarily true if $\hh$ has radical, since a complementary subspace $\mm$ of $\hh$ could be not completely reducible as $\hh$-module.

In order to describe the reductive subalgebras of $\mathfrak{g}_2$, we consider the nondegenerate Hermitian form
\begin{equation}\label{eq_sigma}
\sigma\colon\OO\times\OO\to\CC, \quad \sigma(x,y)= n(x,y)-n({\bf i}x,y){\bf i},
\end{equation}
 and the automorphism 
 $\tau\in\mathrm{Aut}(\OO)$   given by 
 \begin{equation}\label{290922A}
 \tau(q_1+q_2{\bf l})=q_1+({\bf i}q_2){\bf l},
 \end{equation} for $q_i\in\HH$.
Then, as a consequence  of \cite[Theorem~2.1, Corollary~3.5, Proposition~3.6]{LYg2}, we get:

\begin{theorem}\label{th_clasif}
If $\hh$ is a nonabelian proper reductive  subalgebra of the Lie algebra $\mathfrak{g}_2=\der(\OO)$, then either
\begin{itemize}
\item[a)]
$\hh$ is a 3-dimensional simple Lie algebra and $\OO_0$ is an absolutely irreducible $\hh$-module, or \smallskip 
\item[b)]
$\hh$ is conjugated (by an automorphism of $\mathfrak{g}_2$) to one and only one of the subalgebras in the following list:
\begin{itemize}
\item  $\hh_1=\{d\in\mathfrak{g}_2:d(\HH)\subset\HH\}\cong\mathfrak{so}(\HH^\perp,n)\cong\mathfrak{so}(4)$;
\item  $\hh_2=\{d\in\mathfrak{g}_2:d(\HH)\subset\HH,d(\CC)=0\}\cong\mathfrak{u}(\HH^\perp,\sigma)\cong\mathfrak{u}(2)$;
\item  $\hh_3=\{d\in\mathfrak{g}_2:d(\HH)=0\}\cong\mathfrak{su}(\HH^\perp,\sigma)\cong\mathfrak{su}(2)$;
\item  $\hh_4=\{d\in\mathfrak{g}_2:d\tau=\tau d\}\cong\mathfrak{u}(\HH,\sigma)\cong\mathfrak{u}(2)$;
\item  $\hh_5=\mathfrak{cent}_{\hh_1}(\hh_3)=\{d\in\hh_1:[d,\hh_3]=0\}\cong\mathfrak{su}(\HH,\sigma)\cong\mathfrak{su}(2)$;
\item  $\hh_6=\{d\in\mathfrak{g}_2:d(\CC)=0\}\cong\mathfrak{su}(\CC^\perp,\sigma)\cong\mathfrak{su}(3)$;
\item  $\hh_7=\{d\in\mathfrak{g}_2:d(\HH)\subset\HH, \, d({\bf l})=0\}\cong\mathfrak{so}(\HH_0{\bf l},n)\cong\mathfrak{so}(3)$.
\end{itemize}
\end{itemize}
In   case a),   $\mathfrak{g}_2$ is the sum of $\hh$ and an absolutely irreducible $\hh$-module of dimension 11. 
\end{theorem}

We only provide here a rough sketch of the proof jointly with several relevant features to be used later.

\begin{proof}  
A subalgebra $\hh$ of $\mathfrak{g}_2$ turns out to be reductive if and only if $\OO_0$ is a completely reducible $\hh$-module.
The proof in \cite{LYg2}  is based on the fact that, if
$\OO_0$ is not an irreducible $\hh$-module, then it has a submodule, and it can be checked (adding the unit) that a subalgebra isomorphic to either $\HH$ or $\CC$ remains invariant. 

Taking into account that such proof is realized in a more general context, we  add  a comment  here on the precise isomorphisms:
$$
d\in\hh_i\mapsto d\vert_{\HH^\perp}, \quad 
d\in\hh_j \mapsto  (R_{{\bf l}}^{-1}dR_{{\bf l}})\vert_{\HH },\quad 
d\in\hh_6\mapsto d\vert_{\CC^\perp}, \quad 
d\in\hh_7\mapsto d\vert_{\HH_0{\bf l}}, 
$$
for $i=1,2,3$, $j=4,5$. (Recall that $R_{{\bf l}}$ is the right multiplication by ${\bf l}\in\OO$.)
\end{proof}

More explicit descriptions of  the  elements of the subalgebras in Theorem~\ref{th_clasif} in terms of our derivations $D_{x,y}$
can be achieved   as follows.
First, note that the decomposition 
$\OO=\OO_{\bar0}\oplus \OO_{\bar1}$,  for $\OO_{\bar0}=\HH$ and $\OO_{\bar1}=\HH\bf l$, is a $\ZZ_2$-grading on $\OO$. In particular, this $\ZZ_2$-grading induces a 
$\ZZ_2$-grading on the Lie algebra $\der(\OO)$, with homogeneous components:
\begin{equation}\label{230922A}
\begin{array}{l}
\der(\OO)_{\bar0}=\{d\in\mathfrak{g}_2:d(\OO_{\bar i})\subset \OO_{\bar i}\ \forall\, i=0,1\}=\hh_1=D_{\OO_{\bar0},\OO_{\bar0}}+D_{\OO_{\bar1},\OO_{\bar1}},\\
\der(\OO)_{\bar1}=\{d\in\mathfrak{g}_2:d(\OO_{\bar i})\subset \OO_{\bar i+\bar 1}\ \forall\, i=0,1\}=D_{\OO_{\bar0},\OO_{\bar1}}.
\end{array}
\end{equation}
Taking into account  that $\mathfrak{so}(\HH,n)=L_{\HH_0}\oplus R_{\HH_0}\cong 2\mathfrak{su}(2)$ is a sum of two simple ideals, we find that
$$
\der(\OO)_{\bar0}=\hh^l\oplus\hh^r
$$
is also a sum of two simple ideals: $ \hh^l=\{d_a^l:a\in\HH_0\}$ and $\hh^r=\{d_a^r:a\in\HH_0\}$,
where  the derivations $d_a^l$ and $d_a^r$ are determined  by $d_a^l\vert_{\HH^\perp}= R_{{\bf l}}L_aR_{{\bf l}}^{-1} $ and 
$d_a^r\vert_{\HH^\perp}= R_{{\bf l}}R_aR_{{\bf l}}^{-1} $. We can explicitly write down that
\begin{equation}\label{eq_defderivlyr}
d_a^l(q_1+q_2{\bf l})=(aq_2){\bf l},\qquad d_a^r(q_1+q_2{\bf l})=[a,q_1]+(q_2a){\bf l}.
\end{equation}
The indices $l$ and $r$ simply refer to the respective left and right action on the odd part $\HH^\perp$.\footnote{Note that $l$ and $r$ correspond to $-$ and $+$, respectively, in the literature.}
Now,  we have that 
$\hh_3=\hh^l$ and $\hh_2=\hh^l\oplus\langle{d_{\bf i}^r}\rangle$. Also, $[\hh^l,\hh^r]=0$ so we get that $\hh^r=\hh_5$. 
As Eq.~\eqref{eq_corcheteconds} tells that $[d_a^l,D_{p,q}]=0$ for all $p,q\in\HH$, so that $D_{\HH,\HH}\subset\hh_5 $ and, by dimension count, $D_{\HH,\HH}=\hh_5$. More precisely, we have $D_{p,q}=d_{[p,q]}^r$ for any $p,q\in\HH$. Since $d_{\bf i}^l$ commutes with $\tau$, we also obtain  $\hh_4=\hh^r\oplus\langle{d_{\bf i}^l}\rangle$. Finally, we have  $\hh_6=D_{{\bf i},\langle 1,{\bf i}\rangle^\perp}$ and $\hh_7=\{d_a^l-d_a^r:a\in\HH_0\}$, since $d_a^l({\bf l})=a{\bf l}=d_a^r({\bf l}) $. \smallskip

Note that, in Theorem \ref{th_clasif}, the subalgebra of type a) must be a maximal subalgebra of $\mathfrak{g}_2$, since if there were properly contained in another subalgebra, the complementary submodule would be reducible.  Moreover, it corresponds to the so called \emph{principal subalgebra}, which  is related to some important topics in Lie theory. 
We have not provided       an explicit description of such subalgebra in Theorem \ref{th_clasif}, but only of some of the properties which characterize it, because  it is difficult to achieve a concrete description in terms of derivations of the octonions. Such description will be provided in  Proposition~\ref{prop_TDSppal}, where we will define $\hh_8$. 
It does not particularly help  to understand better the related homogeneous space, but we have added it by completeness. However,   its existence (a general fact in Proposition~\ref{ref_compactas}) and its uniqueness up to conjugation (Proposition~\ref{pr_conj})  will be highly relevant in Section~\ref{se_irr} for studying the isotropy irreducible Wolf space. 
Due to the fact that these algebraic questions are not immediate at all, we will specifically devote Section~\ref{se_ppal}  
to deepen in the knowledge of the principal subalgebras, better well-known in the complex case.

\begin{remark}
All semisimple subalgebras of the complex semisimple Lie algebras were determined by Dynkin in 1952 \cite{Dyn2}.
This paper introduces some important concepts, as the \emph{index} of a subalgebra, an integer number which permits to distinguish different (non-conjugate) embeddings of the same algebra. 
As regards $\mathfrak{g}_2^\CC$, the four types of  three-dimensional subalgebras jointly with their indices appear in \cite[Table 16]{Dyn2}.
It also provides the classification of regular subalgebras, which reduces to a combinatorial problem related to root systems, introducing the notions of $R$-subalgebras and $S$-subalgebras, corresponding in some sense to reducible/irreducible subalgebras respectively. 
Our algebra $\mathfrak{g}_2^\CC$ has no any $S$-subalgebra and it has only one simple subalgebra of rank greater than 1 (of type $A_2$), as listed in \cite[Table 25]{Dyn2}.   
Many of these results are summarized and revisited in Chapter~6 in the encyclopaedia \cite{enci41}. Tables~5 and 6 give the two only maximal subalgebras of rank 2 of $\mathfrak{g}_2^\CC$,
 both of them semisimple,  isomorphic to $\mathfrak{sl}_3(\CC)$ and $\mathfrak{so}_4(\CC)$, corresponding to the fixed subalgebra by an inner automorphism of order 3 and 2 respectively.  The results agree with our situation in the real compact case.
 
 In spite of the very thorough study made by Dynkin, it does not contain an
explicit description of all the subalgebras. In the $\mathfrak{g}_2^\CC$-case, such classification appears in a very recent reference: 
 according to \cite[Theorem 1.1.]{subalgebras}, there are 115 subalgebras up to conjugacy by an inner
automorphism: 
 64 types of regular subalgebras,
2 non-regular semisimple subalgebras and 
 49 types of non-regular solvable subalgebras.
The techniques are combinatorial (the paper proceeds by calculation in the Chevalley basis), and differ very much from the techniques in \cite{LYg2}, since \cite{LYg2} studies only the reductive subalgebras (and the only restriction on the ground field is having zero characteristic).
\end{remark}


\subsection{ The principal three-dimensional subalgebra of $\mathfrak{g}_2$}\label{se_ppal}

The subalgebra considered in item a) in Theorem~\ref{th_clasif} cannot be so easily described as the others, that is,  as the subalgebra of $\der(\OO)$ which leaves invariant some subalgebra or commutes with some automorphism.
One could think that this is a weird subalgebra, but the situation is the opposite: this is the most \lq\lq frequent\rq\rq three-dimensional Lie algebra, corresponding to the so called principal subalgebra. Now we will recall these concepts in detail, not only by completeness, but because we will use the knowledge on principal subalgebras for describing the homogeneous space related to this case.
Some of the information is extracted from \cite[Chapter~6, \S2.3]{enci41}. An imprescindible reference   is  
  \cite{Kostant}, where Kostant  relates the 3-dimensional  principal subalgebras   to many other topics.\smallskip

 If $\mathfrak{g}$ is a complex semisimple Lie algebra, the classification of the three-dimensional simple subalgebras of $\mathfrak{g}$ is equivalent to the classification of nilpotent elements. A triple of elements $\{e,h,f\}\subset\mathfrak{g}$ is called an $\mathfrak{sl}_2$-\emph{triple} if $[h,e]=2e$,
$[h,f]=-2f$ and $[e,f]=h$, that is, they form a canonical basis of a subalgebra of $\mathfrak{g}$ isomorphic to $\mathfrak{sl}_2(\CC)$. 
According to Morozov's theorem, for each   nilpotent element $0\ne e\in\mathfrak{g}$, there is some  $\mathfrak{sl}_2$-{triple}  of $\mathfrak{g}$ containing $e$.
The element $h$ in such $\mathfrak{sl}_2$-{triple}   is called the \emph{characteristic} of $e$. It turns out that the set $\mathfrak N=\{e\in\mathfrak{g}:e \textrm{ nilpotent}\}$ is an algebraic variety of dimension equal to $\dim\mathfrak{g}-\mathrm{rank}\,\mathfrak{g}$. The group $G$ of inner automorphisms of $\mathfrak{g}$ acts on $\mathfrak{N}$ producing a finite number of orbits. There is only one dense orbit, open in $\mathfrak{N}$  in the Zarisky topology, which is called the \emph{principal} orbit (being the biggest). The nilpotent elements in this orbit are also called \emph{principal}, and all of them are obviously conjugated. So the condition for a nilpotent element to be principal is that the dimension of $Z(e)=\{\sigma\in G:\sigma(e)=e\}$ coincides with the rank of $\mathfrak{g}$, or equivalently, the dimension of the centralizer $\mathfrak{z}(e)=\{x\in\mathfrak{g}:[x,e]=0\}$ coincides with the rank of $\mathfrak{g}$. The subalgebra spanned by an $\mathfrak{sl}_2$-{triple} $\{e,h,f\}$ where $e$ is a principal nilpotent is called a \emph{principal subalgebra}.

Not every semisimple element is contained in an $\mathfrak{sl}_2$-{triple}  of $\mathfrak{g}$, that is, not every  semisimple element is a 
 characteristic\footnote{If $h$ is a characteristic of $e$, the set of all the characteristics of $e$ is just $\{\sigma(h):\sigma\in Z(e)\}$.}
  of some nilpotent element. If we have fixed $\hh$ a Cartan subalgebra of $\mathfrak{g}$     and   a system of simple roots $\{\alpha_i:i=1,\dots,l\}$ of   the root system relative to $\hh$, then the characteristic of any nilpotent element of $\mathfrak{g}$ is conjugated to some $h\in\hh$ such that $\alpha_i(h)\in\{0,1,2\}$ for all $i$, although the converse is not true.

\begin{example}
In the complex exceptional algebra $\mathfrak{g}_2^\CC$,     
there are $9$ pairs of elements in  $\{0,1,2\}^2$, but not all of them are $(\alpha_1(h),\alpha_2(h))$ for   $h$ a  characteristic of some nilpotent element.  According to \cite{Dyn2} (\cite{orbits}  for the real -of course non compact- case),
 there are just $4$ orbits of nonzero nilpotent elements, corresponding to the pairs
$$
 (0,1),\quad (1,0),\quad (0,2),\quad (2,2).
$$
As the set of positive roots is $ \{\alpha_1,\alpha_2,\alpha_1+\alpha_2,2\alpha_1+\alpha_2,3\alpha_1+\alpha_2,3\alpha_1+2\alpha_2\}$ ($\alpha_1$ short root), an easy computation with eigenvalues shows then that the decomposition of $\mathfrak{g}_2^\CC$ as a sum of $\mathfrak{sl}_2(\CC)$-modules for the corresponding $\mathfrak{sl}_2(\CC)$ is, respectively,
$$
4V(1)\oplus V(2)\oplus 3V(0),\quad
2V(3)\oplus V(2)\oplus 3V(0), \quad
3V(2)\oplus V(4),\quad
V(2)\oplus V(10),
$$
where $V(n)$ denotes here the irreducible $\mathfrak{sl}_2(\CC)$-module of dimension $n+1$. Hence, the corresponding algebras in Theorem~\ref{th_clasif} are (the complexifications of)
$\hh_3$, $\hh_5$, $\hh_7$ and the three-dimensional algebra described in item a). 

Here it is very easy to compute the dimensions of the centralizers $\mathfrak{z}(e)$ and $\mathfrak{z}(h)$. 
For each $V(n)$, the highest vector (that of weight $n+1$) belongs to $\mathfrak{z}(e)$, and no  other independent element, 
while there is at most one vector of weight $0$, which always belongs to $\mathfrak{z}(h)$. Hence  the dimension of $\mathfrak{z}(e)$ coincides with the number of irreducible modules appearing in the decomposition, while the dimension of $\mathfrak{z}(h)$ coincides with the number of irreducible modules $V(n)$ with $n$ even  appearing in the decomposition. In particular $\dim\mathfrak{z}(e)\ge\dim\mathfrak{z}(h)$,\footnote{Take care with the typo in that formula  in \cite[Proposition~2.4]{enci41}.} 
and, in our cases, 
$$
\dim\mathfrak{z}(e)=8,6,4,2;\qquad
\dim\mathfrak{z}(h)=4,4,4,2.
$$
Thus,  the nilpotent element with characteristic $h$ such that $\alpha_1(h)=2=\alpha_2(h)$ belongs to the principal orbit, since $\dim\mathfrak{z}(e)=2$.
\end{example}
This example illustrates the way of getting the characteristic of a principal nilpotent element. In general, we can construct a principal three-dimensional subalgebra as follows:

\begin{lemma}$($\cite[Chapter~6, \S2.3]{enci41}$)$
Let $\hh$ be a Cartan subalgebra of the semisimple complex Lie algebra $\mathfrak{g}$     and  $\{\alpha_i:i=1,\dots,l\}$  a system of simple roots of $\Phi$,  the root system relative to $\hh$. For each $\alpha\in\Phi$, denote by $t_\alpha\in\hh$ the element determined by $\kappa(t_\alpha,t)=\alpha(t)$, being $\kappa$ the Killing form, and by $h_\alpha=\frac{2t_\alpha}{\kappa(t_\alpha,t_\alpha)}$. For any root space $\mathfrak{g}_\alpha$ with $\alpha\in\Phi^+$, and any $0\ne e_\alpha\in\mathfrak{g}_\alpha$,
choose $f_\alpha\in\mathfrak{g}_{-\alpha}$ such that $[e_\alpha,f_\alpha]=h_\alpha$. 
As the Cartan matrix $C=(\langle \alpha_i,\alpha_j\rangle )$ is invertible, take $\{c_j\}_{j=1}^l\subset\CC$ (in fact, subset of $\mathbb Q$) unique scalars such that 
for any $i=1,\dots,l$, the equation
\begin{equation}\label{ces}
  \sum_{j=1}^l\langle \alpha_i,\alpha_j\rangle c_j=2 
\end{equation}
 holds. Then $\{e,h,f\}$ is a principal $\mathfrak{sl}_2$-triple for
$$
e=e_{\alpha_1}+\dots+e_{\alpha_l},\quad
h=c_1h_{\alpha_1}+\dots+c_lh_{\alpha_l},\quad
f=c_1f_{\alpha_1}+\dots+c_lf_{\alpha_l}.
$$
\end{lemma}

We include a proof since the  same argument proves that 
$\{\tilde e,h,\tilde f\}$ is a principal $\mathfrak{sl}_2$-triple too, for
$\tilde e=\gamma_1 e_{\alpha_1}+\dots+\gamma_le_{\alpha_l}$
and $\tilde f=\frac{c_1}{\gamma_1}f_{\alpha_1}+\dots+\frac{c_l}{\gamma_l}f_{\alpha_l}$, and
 for any choice of nonzero scalars $\{\gamma_i:i=1,\dots,l\} $.

\begin{proof}
Equation~\eqref{ces} says that $\alpha_i(h)=2$ for any $i$. The fact $[h,e]=2e$, $[h,f]=-2f$ and $[e,f]=h$ is a straightforward computation. 
Note that $\mathfrak{z}(h)=\hh$ since for $\alpha=\sum m_i\alpha_i$ and  $x\in\mathfrak{g}_{\alpha}$,   then $[h,x]=\alpha(h)x=2(\sum_im_i)x\ne0$.
In general, $\dim\mathfrak{z}(e)\ge\dim\mathfrak{z}(h)$, but both dimensions coincide if all the eigenvalues of $h$ are even. This is just the case
since the set of eigenvalues is $\{2(\sum_im_i):\sum m_i\alpha_i\in\Phi\}$. Hence $\dim\mathfrak{z}(e)=\dim\hh=\mathrm{rank}\,\mathfrak{g}$.
\end{proof}

Come back to our setting, real algebras.   
As far as we know, it is not easy to find many references to principal subalgebras of real Lie algebras. 
We will say that a three-dimensional simple subalgebra (usually denoted by TDS in the literature) 
of a simple compact real Lie algebra $\mathfrak{g}$  is  \emph{principal} if  so is its complexification. 
Such TDS is necessarily isomorphic to $\mathfrak{su}_2$, since  
$\mathfrak{g}$ does not  possess nilpotent elements. As a consequence of the previous arguments,

\begin{proposition}\label{ref_compactas}
Any simple compact real Lie algebra  $\mathfrak{g}$  has a principal three-dimensio\-nal  subalgebra.
\end{proposition}

\begin{proof}
Take, as in the previous lemma, a Cartan subalgebra of the complex Lie algebra $ \mathfrak{g}^\CC\equiv \mathfrak{g}\otimes_{\RR}\CC$, a set of simple roots $\{\alpha_1,\dots,
\alpha_l\}$ and an $\mathfrak{sl}_2$-triple  $\{e_\alpha,f_\alpha,h_\alpha\}\subset \mathfrak{g}^\CC$ adapted to the root decomposition  such that $\alpha(h_\alpha)=2$
for any root $\alpha$. We can assume (see, for instance, \cite[1.3 Theorem]{apsAlb}) that 
$$
\mathbf{i}h_{\alpha_i},e_{\alpha_i}-f_{\alpha_i},\mathbf{i}(e_{\alpha_i}+f_{\alpha_i})\in\mathfrak{g}
$$
for any $i=1,\dots,l$. 
If $C=(\langle \alpha_i,\alpha_j\rangle )$ denotes the Cartan matrix of $ \mathfrak{g}^\CC$, then the coefficients in the inverse $C^{-1}$ are positive \cite[1.2.1. Proposition]{Leites},
and hence all $c_i$ are positive (since the column vector $(c_i)_{i=1}^l$ is twice the vector obtained summing the columns of  $C^{-1}$). For $\gamma_i\in \RR$ such that $\gamma_i^2=c_i$, take $\mathfrak{s}=\mathrm{span}\,\langle x,y,z\rangle\subset\mathfrak{g}
$ for
$$
z:=\sum_{i=1}^lc_i\mathbf{i}h_{\alpha_i},\quad
x:=\sum_{i=1}^l\gamma_i\big(e_{\alpha_i}-f_{\alpha_i}\big),\quad
y:=\sum_{i=1}^l\gamma_i\big(\mathbf{i}(e_{\alpha_i}+f_{\alpha_i})\big).
$$
Now simply note that $\{\tilde e,h,\tilde f\}$ is a principal $\mathfrak{sl}_2$-triple in  $\mathfrak{s}^\CC$, for $h=-\mathbf{i}z$,
$\tilde e=\frac12(x-\mathbf{i}y)$ and $\tilde f=-\frac12(x+\mathbf{i}y)$.
\end{proof}

In our concrete case $\mathfrak{g}=\mathfrak{g}_2$, we can go further and provide an explicit description of a principal subalgebra in terms of our operators $D_{x,y}$, as we did for $\hh_i$, $i=1,\dots,7$. The following is a straightforward computation.
\begin{lemma}\label{le_elh}
The derivation
\begin{equation}\label{hppal}
h:=\frac16\big(4D_{\bf{j},\bf{k}}+5D_{\bf{l},\bf{il}}\big)\in\mathfrak{g}_2
\end{equation}
acts as follows:
$$
\bf{ i}\mapsto 0, \quad
\bf{j }\mapsto \bf{ k}, \quad
\bf{k }\mapsto -\bf{ j}, \quad
\bf{l }\mapsto 2\bf{ il}, \quad
\bf{il }\mapsto -2\bf{l },  \quad
\bf{ jl}\mapsto 3\bf{kl }, \quad
\bf{ kl}\mapsto -3\bf{jl }.
$$
So the eigenvalues of $h\otimes 1\in\mathfrak{g}_2^\CC$ acting on $\OO_0^\CC=\OO_0\otimes_\RR\CC$ are all different, namely, $\{0,\pm\bf{i},\pm2\bf{i},\pm3\bf{i}\}$.
\end{lemma}

Be careful with the confusing notation, since $\bf{i}$ denotes at the same time the element in $\OO_0$ and the scalar in the field $\CC$ that we are using for complexifying.

\begin{proposition}\label{prop_TDSppal}
Take $\hh_8:= \mathrm{span}\,\langle h,x,y\rangle\subset\mathfrak{g}_2$ for $h$ defined as in Eq.~\eqref{hppal} and 
$$
x:=D_{\bf{i},\bf{k}}+\frac{\sqrt{15}}{9}\big(D_{\bf{j},\bf{l}}+D_{\bf{k},\bf{il}}\big),\quad
y:=-D_{\bf{i},\bf{j}}+\frac{\sqrt{15}}{9}\big(-D_{\bf{k},\bf{l}}+D_{\bf{j},\bf{il}}\big).
$$
Then
\begin{equation}\label{eqcuenta}
[h,x]=y,\quad 
[h,y]=-x,\quad 
[x,y]=\frac83h,\quad 
\end{equation}
and $\hh_8 $ is a principal three-dimensional simple subalgebra of $\mathfrak{g}_2=\der(\OO)$.
\end{proposition}

\begin{proof}
The fact that $\hh_8$ is a TDS is a direct consequence of Eq.~\eqref{eqcuenta}, which can be easily checked. The algebra $\hh_8$
is in the situation of item a) in Theorem~\ref{th_clasif}  by taking into account Lemma~\ref{le_elh}, which says that the action of $h $ ($\equiv h\otimes 1$) on $\OO_0^\CC$ is irreducible
($\OO_0^\CC\cong V(6)$ as an $\hh_8^\CC$-module).
\end{proof}

\begin{remark}
This result is directly inspired in a more general and striking not published result   \cite[Teorema~21]{tesis}, which asserts that, for  an arbitrary field of zero characteristic $\mathbb F$, and a Cayley $\mathbb F$-algebra $\mathcal C$, then there is  a three-dimensional simple subalgebra $\mathfrak s$ of $\der(\mathcal C)$ 
such that  $\der(\mathcal C)$ can be decomposed as a sum of 
$\mathfrak s$ with an irreducible $\mathfrak s$-module if and only if there is $c\in \mathcal C_0$ such that $n(c)=15$. 
\end{remark}

\section{Homogeneous spaces of $G_2$ }

Assume $G\times M\to M$ is an action of a  Lie group $G$ on a manifold $M$. The action is said to be transitive if for any points $x,y\in M$, there is $\sigma\in G$ such that $\sigma\cdot x=y$. That is, the action has only one orbit.  In this case, we call $M$ a $G$-\emph{homogeneous} manifold.
For any fixed point $o\in M$ (the origin), the \emph{isotropy subgroup}
$H=\{\sigma\in G:\sigma\cdot o=o\}$  is a closed subgroup of $G$, and  $M$ can be identified with the set of left cosets $G/H$. The natural projection $G\to G/H$ becomes a principal fiber bundle with structure group $H$.
The homogeneous manifold $M\cong G/H$ is reductive if there is an $\mathrm{Ad}(H)$-invariant subspace $\mathfrak{m}$ of $\mathfrak{g}=T_eG$ that is a complement of the Lie subalgebra $\mathfrak{h}$. This condition always implies that $[\mathfrak{h}, \mathfrak{m}]\subset \mathfrak{m}$, and the converse holds whenever $H$ is connected.   The natural projection  $\pi\colon G\to G/H$ is a submersion. Therefore, for a reductive homogeneous manifold $M\cong G/H$ with fixed complement  an $\hh$-module $\mathfrak{m}$, the differential map of $\pi$ at $e\in G$ induces an isomorphism between
$\mathfrak{m}$ and $T_{o}M$.

A Riemannian manifold $(M,g)$ is said to be  homogeneous if the Lie group of all isometries $\mathrm{Isom}(M,g)$ acts transitively. If $G$ is a subgroup
of $\mathrm{Isom}(M,g)$ which also acts transitively,  then the Riemannian manifold $(M,g)$ is said to be $G$-\emph{homogeneous}. In this case, for any fixed point $o\in M$, the  isotropy subgroup is  compact. The \emph{linear isotropy representation} $H\to\mathrm{GL}(T_oM)$ is given by $f\mapsto (f_*)_o$,
where $(f_*)_o$ denotes the differential map of $f$ at the point $o\in M$.

A  connected  Riemann manifold $(M,g)$ is a \emph{symmetric space} if, for any $x\in M$, there is a $g$-isometry $\xi^{x}\colon M\to M$  such that
$\xi^{x}(x)=x$ and $(\xi^{x}_{*})_{x}=-\mathrm{id}_{T_xM}.$ Every symmetric space is a $G$-homogeneous manifold $M\cong G/H$ and the symmetry $\xi^{o}$ gives further structures. There is an involutive automorphism $F\colon G \to G$  
such that $\mathfrak{m}:=\{X\in \mathfrak{g}: (F_{*})_e(X)=-X\}$ is an $\mathrm{Ad}(H)$-invariant subspace of $\mathfrak{g}$ that is a complement of the Lie subalgebra $\mathfrak{h}$. Even more,
the decomposition $\mathfrak{g}=\mathfrak{h}\oplus \mathfrak{m}$ is a $\mathbb{Z}_{2}$-grading with odd part $\mathfrak{m}$, and $\mathfrak{m}$ is endowed with a Lie triple system structure, given by $[x,y,z]=[[x,y],z]$.

In order to be used later, let us denote   by $V_{n,k}$
the Stiefel manifold of all orthonormal $k$-frames in $\mathbb{R}^{n}$ and recall that $\mathrm{dim}\, V_{n,k}=nk-\frac{k(k+1)}{2}$.
The set of all oriented $p$-dimensional subspaces of  $\mathbb{R}^{n}$ is denoted by $\widetilde{\textrm{Gr}}_{n,p}$, and it is known as the Grassmann manifold of the oriented $p$-planes in $\mathbb{R}^{n}$. Its dimension is  $\mathrm{dim}\, \widetilde{\textrm{Gr}}_{n,p}=p(n-p).$
\smallskip

According to \cite[Theorem~3.19]{Warner},  for every Lie group $G$ with Lie algebra $\mathfrak{g}$ and every Lie subalgebra  $\hh$  of $\mathfrak{g}$, there is a unique connected Lie subgroup $H$ of $G$ with corresponding Lie algebra $\hh$. 
Thus, for any $i=1,\dots,8$, let us denote by $H_{i}$ the unique connected Lie subgroup of $G_{2}$ corresponding to every Lie  subalgebra $\hh_i$ in Theorem~\ref{th_clasif}.  
We are in position to give explicit models of each  one of  the reductive homogeneous spaces $G_2/H_i$, for $i=1,\dots,8$. 
We won't be concerned about the homogeneous quotients appearing for not connected subgroups, because they are locally undistinguishable of those ones in our list.


\subsection{$G_2$ as a hypersurface of $V_{7,3}$ } \label{elM0}

Any automorphism $f$ of the octonion algebra is determined by the triple of octonions 
$(f({\bf i}),f({\bf j}),f({\bf l}))$, since the algebra generated by $\{\bf i,\bf j,\bf l\}$ is the whole $\OO$. 
In this way, see  \cite[Remark~5.13]{todoG2}  or \cite[4.1]{Baez},
the group $G_2$ can be identified with the set of \emph{Cayley triples}, that is, triples $(X_0,X_1,X_2)$ of orthonormal vectors in $\RR^7$ such that $X_2$ does not belong to the subalgebra generated by the other two elements, in other words, $\Omega(X_0,X_1,X_2)=0$.
For any Cayley triple $(X_0,X_1,X_2)$, there is a unique automorphism $f\in G_2$ such that $(f({\bf i}),f({\bf j}),f({\bf l}))=(X_0,X_1,X_2)$.
The reason is that   $\{X_0,X_1,X_2,X_0\times X_1,X_0\times X_2,X_1\times X_2,X_0\times (X_1\times X_2)\}$  is an  orthonormal basis of $\RR^7$ and it is easy to reconstruct the image by $f$ of all these basic elements  from $f(X_i)$ for any $i=0,1,2$.
Thus, $G_2$ can be viewed as the following hypersurface inside the Stiefel manifold $V_{7,3}$,
\begin{equation}\label{eq_G}
G_2\cong M_0:=\{(X_0,X_1,X_2)\in V_{7,3}:\Omega(X_0,X_1,X_2)=0\}.
\end{equation} 
For references to this description, see the problem 9c) proposed in \cite[p.~121]{Harvey}. Thus, $M_0$ is the \emph{principal homogeneous space} (homogeneous space   for $G_2$ in which the stabilizer subgroup of every point is trivial) or \emph{torsor} of the group $G_2$.\smallskip

In the remainder of this paper, we provide an explicit geometric description  
of  each one of the  reductive homogeneous spaces $G_{2}/H_{i}$ for $1\leq i\leq 8$.


\subsection{The symmetric space $G_{2}/H_{1}$ } \label{elM1}

We will consider 
$$
M_1:=\{\mathcal Q\le\OO:\dim\mathcal Q=4, \ \mathcal Q^2\subset\mathcal Q\},
$$
 the set of subalgebras of $\OO$ of dimension 4.
The Lie group $G_2$ acts on $M_1$ by $f\cdot \mathcal Q=f(\mathcal Q)$. 
From Remark~\ref{re_subalg}, every subalgebra $\mathcal Q\in M_{1}$ is in the orbit of $\HH$, therefore this action is transitive.
 For the isotropy subgroup of $\HH\in M_1$, we have the isomorphism 
 $$
 \{f\in \mathrm{Aut}(\OO):f(\HH)\subset\HH\}=H_1\to\SO(\HH^\perp,n),\quad f\mapsto f\vert_{\HH^\perp}.
 $$
Hence, we conclude that $M_1\cong G_2/\SO(4)$ and the natural submersion $G_2 \to G_2/\SO(4)$ is given by 
 $$
 G_2\longrightarrow M_1,\qquad f\mapsto f(\HH).
 $$ 
 Alternatively, by means of  the description in Eq.~\eqref{eq_G}, we get
 $$
 M_0\longrightarrow M_1,\qquad (X_0,X_1,X_2)\mapsto\span{1,X_0,X_1,X_0\times X_1}.
 $$ 
 
 A different description of the manifold $M_1$  in terms of  the   3-form $\Omega$ can be provided. 
 \begin{lemma}
 A 4-dimensional vector subspace $\mathcal Q\le\OO  $
 is a subalgebra if and only if  $\mathcal Q^\perp$   is a 4-dimensional subspace of $\OO_0$ where $\Omega$ vanishes.
 \end{lemma}
 
 \begin{proof}
 If $\mathcal Q \in M_{1}$, necessarily $\mathcal Q $ is a subalgebra isomorphic to $\HH$ as in Remark~\ref{re_subalg}.
 Then $\OO=\mathcal Q\oplus \mathcal Q^\perp$ is a $\mathbb Z_2$-grading, $\mathcal Q^\perp\mathcal Q^\perp\subset\mathcal Q$ is orthogonal to $\mathcal Q^\perp$ and the  
  3-form $\Omega$ vanishes on $\mathcal Q^\perp$ (the same happens to $\HH$). 
  
  Conversely, take $W$ a 4-dimensional subspace of $\OO_0$ such that $\Omega(W,W,W)=0$, and let us check that $\mathcal Q=\RR\oplus W^\perp$ is a subalgebra ($\perp$ denotes here the orthogonal subspace in $\OO_0$). 
  Take  $\{X_0,X_1,X_2,X_3\}$ an orthonormal basis of $W$. As $\Omega(X_0,X_1,X_2)=0$, we have that 
  $$
   \{X_0,\, X_1,\, X_2,\, X_0\times X_1,\, X_0\times X_2,\, X_1\times X_2,\, X_0\times (X_1\times X_2)\}  
  $$
  is an  orthonormal basis of $\RR^7$. But $\Omega(X_0,X_1,X_3)=0$ too, so $X_3$ is orthogonal to $X_0\times X_1$ and analogously 
  $X_3$ is orthogonal to $X_0\times X_2$ and to $X_1\times X_2$. This means that $X_3$ should be proportional to the seventh element in the basis, $X_0\times (X_1\times X_2)$, and then $W^\perp=\langle X_0\times X_1,\, X_0\times X_2,\, X_1\times X_2 \rangle$, which is of course closed for the cross product.
 \end{proof}\smallskip
Thus we can consider 
\begin{equation}\label{eq_Msim}
M'_1:=\{W\le\RR^7:\dim W=4,\ \Omega(W,W,W)=0\},
\end{equation}
and the bijective correspondence $M_1\to M'_1$  given by $\mathcal Q\mapsto \mathcal Q^\perp$ is compatible with the $G_2$-action and 
provides an alternative description of the symmetric space $G_{2}/\SO(4)$ which does not make use of the octonionic product.
The submersion in these terms is
$$
 G_2\longrightarrow M_1',\qquad f\mapsto f(\HH\ll).
 $$

\begin{remark}
One of the advantages of this approach is that it can be generalized to other 3-forms (there are two orbits of generic 3-forms in $\RR^7$), providing a family of non-compact manifolds, quotients of the Lie group $\mathrm{Aut}(\OO_s)$, where $\OO_{s}$ denotes the split octonion algebra.
\end{remark}
\begin{remark}
Alternative descriptions appear in the literature. The most usual is considering the 3-dimensional \emph{associative} subspaces, where a 3-dimensional subspace $V$ of $\OO_0$ is said associative if the associator vanishes: $(V,V,V)=0$. This means that $\RR\oplus V$ is necessarily a subalgebra of $\OO$ (isomorphic to $\HH$), hence belonging to $M_1$.
\end{remark}

\begin{remark}
This is the best known quotient of $G_2$ since it is a symmetric space. In fact, let us recall that $G_{2}/H_{1}\cong G_{2}/\SO(4)$ comes from considering the $\ZZ_2$-grading on $\der(\OO)$ induced by the   $\ZZ_2$-grading on $\OO=\HH\oplus \HH\ll$, as in Eq.~(\ref{230922A}). This means that a model for the tangent space is given by the set of odd derivations, $T_\HH M_1\cong\der(\OO)_{\bar1}=D_{\HH_0,\HH\ll}$,  which is a Lie triple system.
An alternative nice description of  this tangent space is found  in \cite{Nakata} as
$$
\{f\colon\HH_0\to\HH\textrm{ linear}:f(\ii)\ii+f(\jj)\jj+f(\kk)\kk=0\}.
$$ 
An explicit  isomorphism of the above vector space with $ \der(\OO)_{\bar1} $ is provided by $d\mapsto f_d$, where $f_d\colon\HH_0\to\HH$ is determined by $d(q)=f_d(q)\ll$ for all $q\in\HH_0$.
As an application, the approach of odd derivations is advantageous because each subtriple provides a totally geodesic submanifold of the symmetric space $G_{2}/\SO(4)$  (see \cite{CartanSimet} and for instance \cite[Chap. 11]{surveyChen}), as in Eq.~\eqref{eq_Ntg} in Section~\ref{elM7}.

The $8$-dimensional symmetric space $G_{2}/\SO(4)$ is a quaternion-K\"ahler symmetric Riemannian manifold,  \cite[Chap. 14]{Besse}. 
This manifold appeared in the classification of  quaternion-K\"ahler symmetric space with non-zero Ricci curvature by Wolf, \cite{Wolf65}. 
After this paper,  quaternion-K\"ahler symmetric spaces are called Wolf spaces. 
\end{remark}


\subsection{The 6-dimensional sphere $G_{2}/\SU(3)$} \label{elM6}

The description of the 6-dimensional sphere $\mathbb{S}^{6}\equiv M_6:=\{X\in\RR^7:n(X)=1\}$ as a quotient of $G_2$ is very well-known too and it is possible to find it in detail in many references (for instance, \cite{todoG2, Harvey,AlbMyung}).  
For the sake of completeness, we briefly recall some details here.
Again identifying $\RR^7$ with $\OO_0$, the action of $G_2=\mathrm{Aut}(\OO)$ restricts to $\mathbb{S}^{6}$ since every automorphism of $\OO$ preserves the norm. The action is transitive. In fact,  any element in $\OO_0$ of norm $1$ can be completed to  a Cayley triple, even more:  any pair of orthonormal vectors in $\RR^7$ can be completed to  a Cayley triple. Now, the map which sends $(\ii,\jj,\ll)$ to a fixed Cayley triple is an automorphism of $\OO$. In particular we find an element of $G_2$ which sends $\ii$ to any norm $1$ element in $\OO_0$, 
which gives  the transitively of the action of $G_2$ on $\mathbb{S}^{6}.$
The isotropy subgroup of the element $\ii\in\mathbb S^6$ is $$H_6=\{f\in\mathrm{Aut}(\OO):f(\ii)=\ii\}\cong\SU(\CC^\perp,\sigma),$$ where $\sigma$ is the Hermitian form 
  in Eq.~\eqref{eq_sigma}, and the precise isomorphism is $f\mapsto f\vert_{\CC^\perp}$. Thus, we   get $\mathbb S^6\cong G_2/\SU(3)$ and the natural submersion reads as
$$
 G_2\longrightarrow  G_{2}/\SU(3)\cong M_6,\qquad f\mapsto f(\ii).
 $$ 
  Alternatively,   in the terms of  Eq.~\eqref{eq_G}, we get
 $$
 M_{0}\longrightarrow  M_6,\qquad (X_0,X_1,X_2)\mapsto X_0.
 $$   
 
 \begin{remark}
 This description of $\mathbb{S}^{6}$  as $G_{2}$-homogeneous manifold is closely related with the nearly K\"ahler structure $J$ induced on $\mathbb{S}^{6}$ from the  cross product $\times$. In fact, the group of automorphisms of this nearly K\"ahler structure on $\mathbb{S}^{6}$ is just $G_{2}$ (see \cite{ABT} for a clear description of these facts with very interesting historical comments).
 
 \end{remark}

 \begin{remark}
 Note that $G_2$ acts also in the projective space $\RR P^6=\RR^7\setminus\{0\}/\sim$, where for any $x,y\in\RR^7\setminus\{0\}$, we say that $x\sim y$ if there is $\lambda\in\RR$ with $x=\lambda y$. The class of $x\in\RR^7\setminus\{0\}$ is denoted by $[x]$. Since the action of $G_2$ on $\RR^7$ is linear,  it induces an action on $\RR P^6$   which is obviously transitive. The only difference is that the isotropy subgroup of $[\ii]\in \RR P^6$ is $\{f\in\mathrm{Aut}(\OO):f(\ii)=\pm\ii\}$, which is not connected, but a double covering of $\SU(3)$  with the same related Lie algebra, $\{d\in\der(\OO):d(\ii)=0\}=\hh_6$. 
 \end{remark}


\subsection{The unit fiber bundle over the six dimensional sphere $G_{2}/\SU(2)^{l}$} \label{elM3}

Since $G_{2}$ is a subgroup of $\SO(\OO_{0}, n)$, we have a natural action  on any Stiefel manifold $V_{7,k}\equiv\{(X_1,\dots,X_k):X_i\in\OO_0,\, n(X_i,X_j)=\delta_{ij}\}$ for any $k\le7$. This action  is    transitive for $k=1,2$, since any   orthonormal $k$-frame  $(X_1,\dots,X_k)$   can be completed to  a Cayley triple. But it  is not transitive for $k=3$, since $G_2$ preserves $\Omega$ and not all the orthonormal $3$-frames behave similarly for $\Omega$. Of course $V_{7,1}\cong\mathbb S^6$ and we study now the Stiefel manifold $V_{7,2}$ as homogeneous quotient of $G_2$. 
 
 The isotropy subgroup of $(\ii,\jj)\in V_{7,2}$ is 
 $$
 H_3=\{f\in\mathrm{Aut}(\OO): f(\ii)=\ii,\, f(\jj)=\jj\}= \{f\in\mathrm{Aut}(\OO):f\vert_{\HH}=\id\}.
 $$
  In a similar way to what happened with its related subalgebra $\hh_3$, we have the isomorphism 
 $$H_3\cong\SU(\HH^\perp,\sigma)=\SU(2),\  f\mapsto f\vert_{\HH^\perp}.$$
 Therefore, we get $V_{7,2}\cong G_2/\SU(2)$.

 \begin{remark}
 We would like to point out that there are several subgroups isomorphic to $\SU(2)$ into $G_{2}$. The related homogeneous manifolds are completely different. 
 Our copy of $\SU(2)$ into $G_{2}$ is achieved by means of Theorem \ref{th_clasif} and \cite[Theorem~3.19]{Warner}.
 A mention to this quotient appears in \cite[p.~121]{Harvey}. 
\end{remark}

The Stiefel manifold $V_{7,2}$ has another geometric interpretation. Namely, for any Riemannian manifold $(M,g)$, the unit fiber bundle over $M$ is described as 
$$\mathcal UM=\{u\in T_pM:g_p(u,u)=1, \ p\in M\}.$$ 
For any $n$, the Stiefel manifold is diffeomorphic to the unit fiber bundle over the sphere
$$
V_{n,2}=\{(X_1,X_2):X_i\in\RR^n,\,\langle X_i,X_j\rangle=\delta_{ij}\}\cong \mathcal U\mathbb S^{n-1},
$$
since $ (X_1,X_2)\mapsto X_2\in T_{X_1}\mathbb S^{n-1}=\span{X_1}^\perp$, which is a unit vector tangent to $X_1\in\mathbb S^{n-1}$.
In particular, the above description as homogeneous manifold of $V_{7,2}$ gives that 
$\mathcal U\mathbb S^6 \cong G_2/\SU(2)$.

  Again,   we can describe explicitly the natural submersions as follows
   \begin{center}  \quad
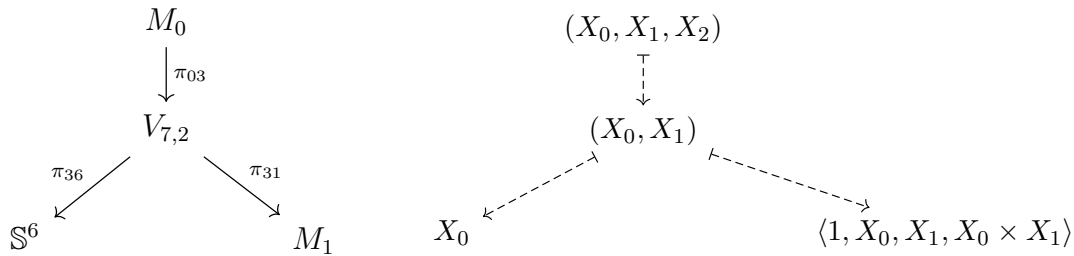
\begin{tikzcd} 
&  M_0\arrow[d,"\pi_{03}"]&\\
&\ar{ld}[swap]{\pi_{36}} V_{7,2}\arrow[rd,"\pi_{31}"]  &  \\  
\mathbb S^6&  &M_1
\end{tikzcd} 
\qquad \small{\begin{tikzcd} 
[arrows={|->}] 
&(X_0,X_1,X_2) \arrow[d,dashed]&\\
&\ar[ld,dashed] (X_0,X_1 )\arrow[rd,dashed]  &  \\  
X_0&  &\span{1,X_0,X_1,X_0\times X_1}
\end{tikzcd}}
\end{center}
or, alternatively,
$
G_2 \to   V_{7,2}\to \mathbb S^6$,  $f\mapsto (f(\ii),f(\jj))\mapsto f(\ii),
$
and in this way 
$
V_{7,2}\to M_1$, $ (f(\ii),f(\jj))\mapsto f(\HH).$

In a general setting, for every Lie group $G$ with closed subgroups $K\subset H\subset G$, the natural projection
$$
G/K \longrightarrow G/H, \quad gK \mapsto gH
$$
is a fiber bundle with standard fiber the homogeneous manifold $H/K.$
In particular, the realizations of $\mathcal{U}\mathbb{S}^{6}$ and of the symmetric space $M_{1}$ as $G_2$-homogeneous manifolds  show that the above projection $
V_{7,2}\to M_1$ is a fiber bundle with standard fiber $\SO(4)/\SU(2)$.


\subsection{ The complex projective quadric $G_{2}/\U(2)^{l}$} \label{elM2}

Consider now the set
$$
M_2:=\{(w,W):W\in M_1,\ w\in W,\ n(w)=1\}.
$$
 Again $f\in G_2$ acts on $M_2$ by means of $f\cdot(w,W)=(f(w),f(W))$. 
 This action is transitive. Namely, for any $(w,W)\in M_2$, take $w'\in W\cap \span{w}^\perp$ such that $n(w')=1$. The element $(w,w')\in V_{7,2}$ can be completed to  a Cayley triple $(w,w',w'')$ with $\Omega(w,w',w'')=0$. Note that $w''\notin W=\span{1,w,w',w\times w'}$ and
 it is clear that the automorphism $f\in\mathrm{Aut}(\OO)$ which sends $(\ii,\jj,\ll)$ to $(w,w',w'')$ satisfies $f(\HH)=W$.
 The isotropy subgroup of the element $(\ii,\HH)\in M_2$ is 
 $$
 H_2= \{f\in\mathrm{Aut}(\OO):f({\HH})\subset \HH,\  f(\ii)=\ii\},
 $$
  whose corresponding Lie algebra is evidently $\hh_2$ in  Theorem~\ref{th_clasif}.
 We also have the isomorphism 
 $$
 H_2\longrightarrow \U(\HH^\perp,\sigma)=\U(2), \quad f\mapsto f\vert_{\HH^\perp},
 $$ 
 and, then, as a consequence of this discussion, we get $M_2\cong G_2/\U(2)$. 
Now, the natural projection $\pi_{26}$ from $M_{2}\cong G_{2}/\U(2)$ to $\mathbb{S}^{6}\cong G_{2}/\SU(3)$  
 is a fiber bundle with standard fiber the homogeneous manifold $\SU(3)/\U(2)$;
 and the natural projection $\pi_{21}\colon M_{2}\to M_{1}$  is a fiber bundle with standard fiber the homogeneous manifold $\SO(4)/\U(2);$
which are given by
  
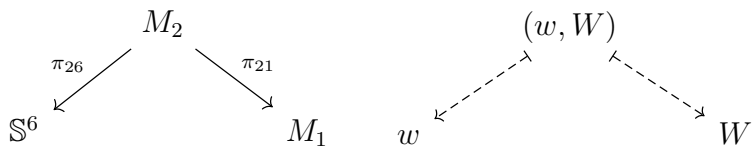
\begin{center}  
  \begin{tikzcd}  
&\ar{ld}[swap]{\pi_{26}} M_2\arrow[rd,"\pi_{21}"]  &  \\   
\mathbb S^6&  &M_1
\end{tikzcd} 
\quad \begin{tikzcd} 
&\ar[ld,dashed,|->] (w,W)\arrow[rd,dashed,|->]  &  \\   w&  &W  
\end{tikzcd}
\end{center}

 \begin{remark}
 Let us note  the appearance of the manifold $M_2$ as $\widetilde{Fl}_{1,ass}(\OO_0)$ in  \cite{NakataNuestrodoblefibrado}. In this paper, F.~Nakata takes into account the   double fibration given by $(\pi_{26},\pi_{21})$ (with our notations)
%
%
  to prove in \cite[Theorem~6.3]{NakataNuestrodoblefibrado} that, for any $w\in\mathbb S_6$, $\pi_{21}({\pi_{26}}^{-1}(w))$ is a totally geodesic submanifold of $M_1$ isomorphic to $\CC P^2$, and for any $W\in M_1$,  $\pi_{26}({\pi_{21}}^{-1}(W))$ is a totally geodesic submanifold of $\mathbb S^6$ isomorphic to $\mathbb S^2$. 
We can also find this double fibration from the point of view of isoparametric hypersurfaces in \cite[Main Theorem]{Miyaoka}. 
\end{remark}

More geometrical interpretations are possible for $M_2$. Let $ \widetilde{\mathrm{Gr}}_{7,2}$ be the Grassmann manifold of the oriented planes in $\RR^7$. This manifold can be identified to $M_2$ by
$$
  \widetilde{\mathrm{Gr}}_{7,2}\longrightarrow M_2, \quad
 \Pi= \span{\{X_1,X_2\}}\mapsto (X_1\times X_2,\span{1,X_1,X_2,X_1\times X_2}),
$$
where $\{X_1,X_2\}$ is an oriented orthonormal basis of $\Pi \in \widetilde{\mathrm{Gr}}_{7,2}.$ There is a natural action of $G_{2}$ on the Grassmann manifold $\widetilde{\mathrm{Gr}}_{7,2}$ and the above identification is compatible with the $G_2$-actions.
The inverse map sends $(w,W)\in M_2$ to the plane $  W\cap \span{1,w}^\perp$ with the orientation given by 
$\span{\{X,w\times X\}}$ for any $X\in  W\cap \span{1,w}^\perp$ with $n(X)=1$. For instance $(\ii,\HH)\in M_2$ corresponds to the oriented plane $\span{\{\jj,\kk\}}\in \widetilde{\mathrm{Gr}}_{7,2}$. \smallskip

  Let us recall the well-known diffeomorphism of the Grassmann manifold $ \widetilde{\mathrm{Gr}}_{7,2}$ 
with the following quadric of the complex projective space, 
$$
Q_5=\{[z]\in\CC P^6=\frac{\CC^7\setminus\{0\}}{\sim}:z_1^2+\dots+z_7^2=0\}.
$$
The identification proceeds as follows
$$
  \widetilde{\mathrm{Gr}}_{7,2}\longrightarrow Q_5, \quad
  \span{\{X_1,X_2\}}\mapsto  [X_1+\ii X_2],
$$
and the inverse map sends every $[z]\in Q_5$ to the plane oriented by $\span{\{\mathrm{Re}(z),\mathrm{Im}(z)\}}\le\RR^7$. \smallskip
Hence, we have $Q_{5}\cong G_{2}/\U(2).$

\begin{remark}
Every complex quadric $Q_{m}$ inherits a Riemannian metric from the Fubini-Study metric on $\mathbb{C}P^{m+1}$. Thus, each of these complex hyperquadrics  $Q_m$ is a symmetric space, but viewed as $Q_m\cong\SO(m + 2)/\SO(2) \times\SO(m)$,
 which is not the decomposition considered here for $m=5$. Recall that the same situation happened for the spheres $\mathbb S^m\cong\SO(m + 1)/\SO(m)$ and $m=6$. 
 By taking advantage of the techniques of Lie triple systems, this description of   $Q_m$ as a symmetric space  has been used to obtain its  totally geodesic submanifolds,  \cite{quadrics}. For $m=5$, the diffeomorphism $Q_{5}\cong G_{2}/\U(2)$ was used by R.~Bryant in \cite{Bryant}. 
Recall that A.~Gray  proved in \cite{GrayS6} that
 every almost complex submanifold of the nearly K\"ahler manifold $\mathbb{S}^{6}$ is minimal and $\mathbb{S}^{6}$
 has no $4$-dimensional almost complex submanifolds with respect to this nearly K\"ahler  structure.
In this context, R~Bryant investigated almost complex curves in $\mathbb{S}^{6}$, that is, non-constant smooth maps $f \colon M^2 \to \mathbb{S}^{6}$ from a Riemann surface $M^2$ such that the differential map $f_{*}$ is complex linear.
Every complex curve is minimal and the ellipse of curvature
$
\{\textrm{II}(u,u): u\in \mathcal{U}_{p}M^{2}\}
$
describes a circle in $(T_{p}M^{2})^{\perp}$ for every $p\in M^{2}$, where $\textrm{II}$ is the second fundamental form.
The  almost complex curve $f \colon M^2 \to \mathbb{S}^{6}$ is called superminimal when morever $
\{(\bar{\nabla}_{u}\textrm{II})(u,u): u\in \mathcal{U}_{p}M^{2}\}
$
  describes a circle too, where $\bar{\nabla}$ denotes the Levi-Civita connection of $\mathbb{S}^{6}$. In order to study superminimal almost complex curves in $\mathbb{S}^{6}$,   Bryant construction in \cite{Bryant} just uses the natural projection 
$$
Q_{5}\cong G_{2}/\U(2) \longrightarrow \mathbb{S}^{6}\cong G_{2}/\SU(3).
$$
See  \cite[Sec. 19.1]{surveyChen} and references therein for more details.
\end{remark}


\subsection{ The unknown quotient $G_{2}/\SO(3)$} \label{elM7}

We consider now the set $$\{(w,W):W\in M_1,\ w\in W^\perp,\ n(w)=1\},$$ which is in one-to-one correspondence with
$$
M_7:=\{(w,W):W\in M'_1,\ w\in W,\ n(w)=1\}
$$
by means of the map $(w,W)\to (w, W^{\perp}).$
This manifold has a quite similar description to the above one of $M_2$, but its geometry has nothing to do, not even the dimension.
Again $f\in G_2$ acts on $M_7$ as   $f\cdot(w,W)=(f(w),f(W))$.

This action of $G_2$ on $M_7$ is transitive. Indeed,   take  $(\ll,\HH\ll)\in M_7$. For any $(w,W)\in M_7$, let $X_1,X_2$ be a pair of orthonormal vectors in $W^\perp$ (in $W^\perp\cap\span{1}^\perp$ if we think of the orthogonal in $\OO$ instead of in $\RR^7$). Thus,  $(X_1,X_2,w)$ is a Cayley triple. The automorphism $f\in\mathrm{Aut}(\OO)$ which sends $(\ii,\jj,\ll)$ to $(X_1,X_2,w)$ satisfies $f(\HH)=\RR\oplus W^\perp$,
because $\{X_1,X_2\}$ is a set of generators (as an algebra) of the subalgebra $\RR\oplus W^\perp$.
Hence $(f(\ll),f(\HH\ll))=(w,W)$. On the other hand,
recall that an automorphism $f$ leaves $\HH$ invariant if and only if it leaves $\HH^\perp=\HH\ll$ invariant, so that 
the isotropy subgroup of $(\ll,\HH\ll)\in M_7$ is 
$$
H_7= \{f\in\mathrm{Aut}(\OO):f({\HH})\subset \HH, \ f(\ll)=\ll\}, 
$$
whose related Lie algebra is evidently $\hh_7$. We have the isomorphism
$$
H_7 \longrightarrow  \SO(\HH_0\ll,n)=\SO(3), \quad f\mapsto f\vert_{\HH_0\ll},
$$  given  by the restriction map. 
Therefore, $M_7$ is a manifold diffeomorphic to $  G_2/\SO(3)$,  and the natural submersion reads as
 $$
 G_2 \longrightarrow  M_7,\quad f\mapsto (f(\ll),f(\HH\ll)).
 $$
 Alternatively, through   Eq.~\eqref{eq_G}, we get
$$
 M_0 \longrightarrow M_7, \quad
   (X_0,X_1,X_2)\mapsto (X_0,\span{X_0,X_0\times X_1,X_0\times X_2,X_0\times (X_1\times X_2)} ).
$$
The manifold $M_{7}$ is a fiber bundle over the sphere $\mathbb{S}^{6}$ and over the symmetric space $G_{2}/H_{1}$.
The projections maps  can be again  described as follows
\begin{equation} \label{280922A}
\centering
  \begin{tikzcd}  
&\ar{ld}[swap]{\pi_{76}} M_7\arrow[rd,"\pi_{71}"]  &  \\   
\mathbb S^6&  &M'_1
\end{tikzcd} 
\quad \begin{tikzcd} 
&\ar[ld,dashed,|->] (w,W)\arrow[rd,dashed,|->]  &  \\   w&  &W  
\end{tikzcd} 
\end{equation} 
%
The standard fibers of $\pi_{76}$ and $\pi_{71}$ are  $\SU(3)/\SO(3)$ and $\SO(4)/\SO(3)\cong\mathbb S^3$, respectively.

\begin{remark}
As far as we know, this space $M_7\cong G_2/\SO(3)$  appears in the literature less frequently  than the other $G_2$-homogeneous manifolds. The only appearance that we have found, occurs in \cite{Enoyoshi}  and \cite{Enoyoshiprevio} as the set $\tilde M(0)$ of level $t=0$ of  the calibration $\Omega$, 
$$\tilde M(t)=\{V\in \widetilde{\mathrm{Gr}}_{7,3}: \Omega(V,V,V)=t\}.$$
Enoyoshi first proved in \cite{Enoyoshiprevio} that all these sets $M(t)$ are diffeomorphic to $G_2/\SO(3)$ for $t\in(-1,1)$ and he also computes their principal curvatures as hypersurfaces of $ \widetilde{\mathrm{Gr}}_{7,3}$. Remarkably, the level sets $M(1)$ and $M(-1)$ are diffeomorphic to the symmetric space $G_{2}/H_{1}$. 
The double fibration (\ref{280922A}) appears in \cite{Enoyoshi}  to prove
that,  for any $w\in\mathbb S_6$, the submanifold $\pi_{71}({\pi_{76}}^{-1}(w))$ is  totally geodesic in $M'_1$ and isomorphic to $\SU(3)/\SO(3)$, and for any $W\in M'_1$, the submanifold  $\pi_{76}({\pi_{71}}^{-1}(W))$ is isomorphic to $\mathbb S^3$ and  totally geodesic  and Lagrangian in $\mathbb S^6$ \cite[Theorem~4.2]{Enoyoshi}. Let us recall that a submanifold of $\mathbb{S}^6$ is Lagrangian when the nearly K\"ahler structure of $\mathbb{S}^{6}$ applies the tangent bundle of the submanifold in its normal  bundle. 
 \end{remark}

 The existence of these 5-dimensional totally geodesic submanifolds of $M_1'$ agrees with the results of Klein \cite[Theorem~5.4]{Kleintg} on maximal totally geodesic submanifolds of  the exceptional Riemannian symmetric spaces of rank~$2$, obtained with very different arguments based on Lie triple systems. As we have concrete expressions of $\pi_{76}$ and $\pi_{71}$, we can easily compute $\pi_{71}({\pi_{76}}^{-1}(\ll))$, which coincides with
\begin{equation}\label{eq_Ntg}
N=\{W\in M'_1:\ll\in W\}.  
\end{equation}
So $N$ is a totally geodesic submanifold of $M_1'$ diffeomorphic to $\SU(3)/\SO(3)$. 
An independent algebraic proof is given by the fact that
 the $5$-dimensional vector subspace 
$$
\mathfrak{n}=\span{D_{p,p\ll}:p\in\HH_0}\le D_{\HH,\HH\ll}=\der(\OO)_{\bar 1}
$$
 is closed for the triple product $[x,y,z]=[[x,y],z]$. 
 Besides $\{f\in G_2:f(\ll)=\ll\}\cong\SU(3)$ (isomorphic to $H_6$) acts transitively on $N$ and the isotropy subgroup of $\HH\ll\in N$ is of course 
$H_7=\{f\in G_2:f(\HH)\subset\HH, f(\ll)=\ll\}\cong\SO(3)$.
Note that, in turn, the manifold $N$ is a symmetric space which has received quite attention  in \cite{Nurowski}, where five-dimensional geometries modeled on $N$
have been studied. \smallskip

An algebraic  model for the tangent space of $G_{2}/\SO(3)$ can be found in \cite[\S6]{ModelosG2}, mainly based on linear algebra. Some work in progress is making use of this model in order to find  good metrics in $M_7$.


\subsection{ The twistor space of complex structures $G_{2}/\U(2)^{r}$} \label{elM4}

As   mentioned above, the symmetric space $M_1$ is an $8$-dimensional well-known quaternion-K\"ahler
manifold (see, for instance, \cite[Chap. 14]{Besse} for general notions on quaternion-K\"ahler
manifolds). In particular, there is a subbundle $Q\subset \mathrm{End}\,TM_1$, locally generated by three anticommuting fields of
endomorphisms $J_1$, $J_2$, $J_3=J_1J_2$ such that $J_i^2=-\mathrm{id}$, which consists of 
skewsymmetric
endomorphisms
and such that the Levi-Civita connection preserves $Q$. 
The  twistor space 
$$
\mathcal Z=\{0\ne A\in Q:A^2=-\mathrm{id}\},
$$
 endowed with the natural projection on $M_1$, is a fiber bundle with standard fiber the two-dimensional sphere $\mathbb{S}^2$. It is a well-known space, for instance $\mathcal{Z}$ has a natural complex structure  \cite[Theorem~4.1]{QK} and  a K\"ahler-Einstein metric of positive scalar curvature. 
We now describe   the twistor space $\mathcal{Z}$ of the quaternion-K\"ahler symmetric manifold $M_1'$ from our concrete algebraic viewpoint.

Recall that, for $W$ a real vector space, an endomorphism $J\colon W\to W$ is said to be a \emph{complex structure} of $W$ if $J^2=-\id_W$. 
In this case, $W$ can be endowed with a complex vector space structure by taking as a scalar multiplication $\CC\times W\to W$, $(\alpha+\ii\beta)w=\alpha w+J(\beta w)$. 
If we have fixed a scalar product on $W$, we will say that
the complex structure $J$ is \emph{metric}   if $J\in\SO(W)$\footnote{Note that, if $J^2=-\mathrm{id}$, then $J\in\SO(W)$ if and only if $J\in\mathfrak{so}(W)$.}.
We consider$$
N_4=\{(W,J):W\in M_1',\  J\textrm{ metric complex structure on }W\},
$$
where the subspaces $W$ of $\RR^7$ are endowed with the scalar product inherited from the usual one on $\RR^7$.
It is evident that $G_2$ acts on $N_4$ by $f\cdot(W,J)=(f(W),fJf^{-1})$. The derivation $d_{\ii}^l$, considered in Eq.~\eqref{eq_defderivlyr}, restricts to a complex structure $d_{\ii}^l\vert_{\HH^\perp}$ of $\HH^\perp$, since $q\ll\mapsto(\ii q)\ll\mapsto -q\ll$. Besides, $d_{\ii}^l$ is metric since $n(q\ll)=n((\ii q)\ll)$.
In other words, we have $(\HH\ll,d_{\ii}^l)\in N_4$.
The subgroup of $G_2$ which fixes this element in $  N_4$ is
$$
H_4=\{f\in\mathrm{Aut}(\OO):f({\HH}^\perp)\subset \HH^\perp,\ fd_{\ii}^l\vert_{\HH^\perp}=d_{\ii}^lf\vert_{\HH^\perp}
\}.$$

\noindent Let us check that $H_4$ coincides with $\{f\in\mathrm{Aut}(\OO):f\tau=\tau f\}$, the unique connected Lie subgroup with related Lie algebra   $\hh_4$. 
(Recall that  the automorphism $\tau$ was defined in Eq.~(\ref{290922A}).)
Consider $f\in H_4$,  and take into account that $f(\HH)\subset \HH$ and $\tau\vert_\HH=\id$ to get $f\tau=\tau f$ on $\HH$. 
Also, we    have $f\tau=\tau f$ in $\HH^\perp$, since 
$d_{\ii}^l\vert_{\HH^\perp}=\tau\vert_{\HH^\perp}$ (funny fact, since $d_{\ii}^l\in\der(\OO)$ but $\tau\in \mathrm{Aut}(\OO)$). Hence $f\tau=\tau f$. 
Conversely, if $f$ is an automorphism commuting with $\tau$, then $f(\HH)\subset \mathrm{Fix}(\tau)=\HH$ and hence $f({\HH}^\perp)\subset \HH^\perp$. Again the fact $d_{\ii}^l\vert_{\HH^\perp}=\tau\vert_{\HH^\perp}$ finishes the discussion. In particular the restriction $f\mapsto f\vert_\HH$ provides the isomorphism $H_4=\{f\in\mathrm{Aut}(\OO):f\tau=\tau f\}\cong \U(\HH,\sigma).$ In order to study whether the action of $G_2$ on $N_4$ is  transitive or not, the results in the following lemma are useful.

\begin{lemma}\label{le_M4}
Fix $(W,J)\in N_4$. 
\begin{itemize}
\item[a)] For any $X\in W$ with $n(X)=1$ and any $Y\in W\cap\span{X,J(X)}^\perp$ with $n(Y)=1$, we have
\begin{equation}\label{eq_W}
W=\span{X,Y,J(X),J(Y)}, \quad W^\perp=\span{X\times Y,X\times J(X),Y\times J(X)}.
\end{equation}
\item[b)] For $X,Y$ as above, 
 there is $\alpha\in\{\pm1\}$ such that 
 $$
 (X\times Y)\times J(X)=\alpha J(Y).
 $$
If $ \alpha=1$, the automorphism of $\OO$ which sends $(\ii,\ll,\jj\ll)$ to the Cayley triple $(X\times J(X),X,Y)$   satisfies that $f\cdot(\HH\ll,d_{\ii}^l)=(W,J)$. 
\end{itemize}
\end{lemma}

\begin{proof}
Recall (for instance, \cite[Eqs.~(3.2), (4.13) and (4.14)]{todoG2}) that the cross product is anticommutative and satisfies
\begin{equation}\label{eq_util}
XY=X\times Y,\quad  X\times (X\times Y)=-n(X)Y,\quad X\times (Y\times Z)=-Y\times (X\times Z),
\end{equation}
whenever $X,Y,Z\in\RR^7\equiv\OO_0$ are pairwise orthogonal.
Since $J\in\SO(W,n)$ and $J^2=-\id$, then  $n(X,J(X))=n(J(X),J^2(X))=-n(X,J(X))$, so that $J(X)$ is orthogonal to $X$ for any $X\in W$.
\begin{itemize}
\item[a)]  It is clear that we can choose $Y\in W\cap\span{X,J(X)}^\perp$ with $n(Y)=1$. Let us check that $\{X,Y,J(X),J(Y)\}$ is an orthonormal frame in $W$, in particular a basis. Indeed, $n(J(Y))=n(Y)=1$ and $J(Y)$ is orthogonal to $Y$ as  above. Also, we have $n(J(Y),X)=n(J^2(Y),J(X))=-n(Y,J(X))=0$ and $n(J(Y),J(X))=n(J^2(Y),J^2(X))=n(-Y,-X)=0$. As $W\in M'_1$, we have $\Omega(W,W,W)=0$ and so $W\times W\subset W^\perp$ (orthogonal in $\RR^7\equiv\OO_0$). 
To get \eqref{eq_W}, we only need to prove that 
$\{X\times Y,X\times J(X),Y\times J(X)\}$ are linearly independent, for instance checking that they are orthogonal. This is straightforward:
$$n(X\times Y,X\times J(X))=n(XY,XJ(X))=n(X)n(Y,J(X))=0$$  and the same argument applies  to the other two cases.

\item[b)] Clearly,  $(X\times J(X),X,Y)$ is a Cayley triple:
$$
\Omega(X\times J(X),X,Y)=n(X\times J(X),X\times Y)=n(X)n(J(X),Y)=0,
$$
 so that we can take  the automorphism $f\in\mathrm{Aut}(\OO)$
 determined by $f(\ii)=X\times J(X)$, $f(\ll)=X$ and $f(\jj\ll)=Y$. Now we use Eq.~\eqref{eq_util} to compute the images under $f$:
$$
\begin{array}{ll}
\ii\ll\mapsto J(X),&
\kk\ll=\jj(\ii\ll)\mapsto  (X\times Y)\times J(X),\\
\jj=-(\jj\ll)\ll\mapsto X\times Y,\qquad\qquad &
\kk=-(\ii\ll)(\jj\ll)\mapsto Y\times J(X).
\end{array}
$$
Note that the element $(X\times Y)\times J(X)\in (W\times W)\times W\subset W$ is orthogonal to $X$, $Y$ and $J(X)$. Therefore, there is $\alpha\in\RR$ such that
$(X\times Y)\times J(X)=\alpha J(Y)$. Thus     
$$
(f(\ll),f(\ii\ll),f(\jj\ll),f(\kk\ll))=(X,J(X),Y,\alpha J(Y)),
$$
and, in particular,  $f(\HH^\perp)=W$.  Taking norms, $\alpha^2=n(X)n(Y)n(J(X))=1$ and necessarily $\alpha=\pm1$.
In case $\alpha=1$, we immediately check that 
$$
\begin{array}{ll}
f\tau(\ll)=J(X)=Jf(\ll), \qquad &f\tau(\ii\ll)=-X=J^2(X)=Jf(\ii\ll),\\
 f\tau(\jj\ll)=J(Y)=Jf(\jj\ll), \qquad\  & f\tau(\kk\ll)=-Y=Jf(\kk\ll).
 \end{array}
 $$
Hence, we get $f\tau\vert_{\HH^\perp}=Jf\vert_{\HH^\perp}$, as required.
\end{itemize}
\end{proof}

Note that $J=d_{\ii}^l\vert_{\HH^\perp}$ satisfies the following additional property:
$$
J(X)\times J(Y)=((\ii q)\ll)((\ii p)\ll)=-(\overline{\ii p})(\ii q)=-\bar p\bar\ii\ii q=-\bar pq=(q\ll)(p\ll)=X\times Y,
$$
for any $X=q\ll$, $Y=p\ll\in\HH\ll=\HH^\perp$. Nevertheless,  $(\HH\ll,d_{\ii}^r)\in N_4$ does not preserve  the cross product $W\times W\subset
W^\perp$. This tells   
 that the  action of $G_2$ on $N_4$ is not transitive and  leads us to consider
$$
M_4:=\{(W,J):W\in M_1', \,J\in\SO(W,n),\, J^2=-\id,\, J(X)\times J(Y)=X\times Y\  \forall X,Y\in W\}.
$$
Since $G_2$ preserves $\times$, the Lie group $G_2$   acts on $M_4$ too,  but now this action is transitive.
Indeed, if $(W,J)\in M_4$, take $X$ and $Y$ as in Lemma~\ref{le_M4}. Then
$$
(X\times Y)\times J(X)=(J(X)\times J(Y))\times J(X)\stackrel{\eqref{eq_util}}{=}J(Y)
$$ 
and item b) allows to find a concrete automorphism $f\in G_2$ with $f\cdot(\HH\ll,d_{\ii}^l)=(W,J)$. Hence, taking into account that the isotropy group at $(\HH\ll,d_{\ii}^l)$
is $H_4\cong \U(\HH,\sigma)\cong \U(2)$, we get $M_4\cong G_2/\U(2)$.
 According to \cite[Chap. 13]{libroBG}, the manifold  $M_4$ can be identified with the twistor space $\mathcal{Z}$ of the quaternion-K\"ahler manifold $G_{2}/\SO(4)$.

In order to write the \emph{good} complex structures involved in the definition of $M_4$, without reference to an outer object, let us consider for any $W\in M_1'$ the ternary product
given by
$$\{\ ,\ ,\ \}\colon W\times W\times W\to W, \quad\{X,Y,Z\}:=(X\times Y)\times Z.$$
Now, we consider
$$
\mathrm{Aut}(W,n,\{\ ,\ ,\ \}):=\{J\in\SO(W,n):\{J(X),J(Y),-\}=\{X,Y,-\}\  \forall X,Y\in W\},
$$
and then it is easy to check that 
$$
M_4=\{(W,J):W\in M_1', \, J\in\mathrm{Aut}(W,n,\{\ ,\ ,\ \}),\ J^2=-\id\}.    
$$


\subsection{ The twistor space of quaternionic structures $G_{2}/\SU(2)^{r}$} \label{elM5}
For any real vector space $W$, a pair $(J,K)$ of endomorphisms  $J,K\colon W\to W$ is said to be a \emph{quaternionic structure} on $W$ if $J^2=K^2=-\id_W$ and $JK=-KJ$. In other words, $J$, $K$ and $JK$ are three complex structures which anticommute.
In this case, $W$ can be endowed with a quaternionic vector space structure by taking as a scalar multiplication $W\times \HH\to W$, 
$w(\alpha+\beta\ii+\gamma\jj+\delta\kk):=\alpha+ \beta KJ(w)+\gamma J(w)+\delta K(w)$.

The group $G_2$ acts on the set 
$$
M_5:=\left\{
\begin{array}{rl}
(W,(J,K)):&W\in M_1'  \\
&(J,K)\textrm{   quaternionic structure on }W\\
&J,K\in\mathrm{Aut}(W,n,\{\ ,\ ,\ \})
\end{array}\right\}.
$$
The action of $f\in\mathrm{Aut}(\OO)$ is given by $f\cdot(W,(J,K))=(f(W),(fJf^{-1},fKf^{-1}))$,   well defined since $f$ preserves the ternary product $\{\ ,\ ,\ \}$.
As $d_a^ld_b^l(q_1+q_2\ll)=(abq_2)\ll$, then we have $d_a^ld_b^l\vert_{\HH^\perp}=d_{ab}^l\vert_{\HH^\perp}$ and therefore $(\HH^\perp,(d_\ii^l,d_\jj^l))\in M_5$. The subgroup of $G_{2}$ which fixes the element $(\HH^\perp,(d_\ii^l,d_\jj^l))\in M_5$ is
$$
H_5=\{f\in\mathrm{Aut}(\OO):f({\HH}^\perp)\subset \HH^\perp,
\ fd_{\ii}^l\vert_{\HH^\perp}=d_{\ii}^lf\vert_{\HH^\perp},
\ fd_{\jj}^l\vert_{\HH^\perp}=d_{\jj}^lf\vert_{\HH^\perp}\}.
$$
Note that every element $f\in H_{5}$ commutes with $d_{\kk}^l\vert_{\HH^\perp}=d_{\ii}^l\vert_{\HH^\perp}d_{\jj}^l\vert_{\HH^\perp}$ too,  so that  the Lie algebra of $H_5$  is
$$
\{d\in\der(\OO)_{\bar0}:dd_{a}^l\vert_{\HH^\perp}=d_{a}^ld\vert_{\HH^\perp}\  \forall a\in\HH_0\}=\{d\in\der(\OO)_{\bar0}:[d,\hh^l]=0\}=\hh_5,
$$
because of course any $d\in\der(\OO)_{\bar0}$ commutes with $d_{a}^l\vert_{\HH}\equiv0$. 
This tells us that we have the isomorphism
$$
H_5\cong \SU(\HH,\sigma), \quad f\mapsto f\vert_\HH.
$$

The transitivity of the action of $G_2$ on $M_5$ is a consequence of the following lemma.

\begin{lemma}\label{le_M5}
Fix $(W,(J,K))\in M_5$. 
\begin{itemize}
\item[a)] For any $X\in W$ with $n(X)=1$, we have
\begin{equation*}\label{eq_W5}
W=\span{X,J(X),K(X),JK(X)}, \quad W^\perp=\span{X\times J(X),X\times K(X),X\times JK(X)};
\end{equation*}
\item[b)] The automorphism of $\OO$ which sends $(\ii,\jj,\ll)$ to the Cayley triple $(X\times J(X),X\times K(X),X)$ satisfies that $f\cdot (\HH^\perp,(d_\ii^l,d_\jj^l))=(W,(J,K))$. 
\end{itemize}
\end{lemma}

\begin{proof}
In order  to check that $\{X,J(X),K(X),JK(X)\}$ is a basis of $W$, it is enough to check that it is an orthonormal set. Indeed, the four elements have norm 1 since 
$J,K,JK\in\SO(W,n)$; and $X$ is orthogonal to the others
as in Lemma~\ref{le_M4}. For another pair of elements, $n(J(X),K(X))=n(J^2(X),JK(X))=-n(X,JK(X))=0$. 
Besides, as  $\Omega(W,W,W)=0$, then $W\times W\subset W^\perp$ and we get a).

Retain Eq.~\eqref{eq_util} for the remaining computations. We have a Cayley triple 
$$
\Omega(X,X\times J(X),X\times K(X))=n(-J(X),X\times K(X))\in n(W,W^\perp)=0,
$$
 so we can consider
 the automorphism $f\in\mathrm{Aut}(\OO)$ determined by $f(\ii)=X\times J(X) $, $f(\jj)=X\times K(X) $, $f(\ll)=X$. The rest of images necessarily are
$$
\begin{array}{ll}
f(\ii\ll)=(X\times J(X))\times X=J(X),&
f(\jj\ll)=K(X),\\
f(\kk)=f((\jj\ll)(\ii\ll))=K(X)\times J(X),&
f(\kk\ll)=X\times (J(X)\times K(X)).
\end{array}
$$
Let us check that $f(\kk\ll)=JK(X)$. Of course there is $\alpha\in\RR$ such that $f(\kk\ll)=\alpha JK(X)$, 
since $X\times (J(X)\times K(X))\in W$ is orthogonal to $X$, $J(X)$ and $K(X)$.
In particular, we have $f(\HH\ll)=W$.
Taking norms gives $\alpha\in\pm1$, but, using besides that $J\in\mathrm{Aut}(W,n,\{\ ,\ ,\ \})$ we can conclude $\alpha=1$ as a consequence of the following computation:
$$
J(X)\times K(X)=-X\times \big(X\times (J(X)\times K(X))\big)=
-X\times \alpha JK(X)=-J(X)\times \alpha J^2K(X).
$$
Having established   $f(\kk\ll)= JK(X)$,   it is easily deduced that $fd_\ii^l=Jf$ and $fd_\jj^l=Kf$ in $\HH^\perp$.
\end{proof}

Hence, as  $H_5\cong \SU(2)$, we have  that $M_5\cong G_2/\SU(2)$ 
and the natural projections can be explicitly given by
$$
M_5\longrightarrow M_4\longrightarrow M'_1,\qquad (W,(J,K))\mapsto (W,J)\mapsto W,
$$
which are fiber bundles with standard fiber $\U(2)/\SU(2)$ and $\SO(4)/\U(2)$, respectively.
There is also a natural projection $M_0\to M_5$ that sends $(X_0,X_1,X_2)\in M_0$ to $(W,(J,K))$ for
$$
W=\span{X_0,X_0\times X_1,X_0\times X_2,X_0\times(X_1\times X_2)},\quad  J=L_{X_1},\ K=L_{X_2},
$$
and $L_X\colon W\to W$  given by $L_X(Y)=X\times Y$ if $X\in W^\perp$.

 \begin{remark}
 
 The quotient $G_{2}/\SU(2)^{r}$ has a very rich geometric structure. Namely, $G_{2}/\SU(2)^{r}$ is a $3$-Sasakian homogeneous manifold \cite[Chap. 13]{libroBG}. A general study of invariant linear connections on $3$-Sasakian homogeneous manifolds can be found in \cite{nues3Sas}. At this point, we would like to do a  personal remark.  The manifold $G_{2}/\SU(2)^{r}$ played a key role in the motivation of \cite{nues3Sas}
 and its algebraic structure inspired the notion of $3$-Sasakian data there.
\end{remark}


\subsection{ The irreducible isotropy} \label{se_irr}  

A $G$-homogeneous manifold $M=G/H$ is called an isotropy-irreducible space if the linear isotropy representation of $H$ is irreducible. 
In this case, Wolf proved in \cite{Wolf68} that $M$ admits a unique $G$-invariant Riemannian metric (up to homotheties), which is necessarily an Einstein metric. In the same paper,
Wolf classified the $G$-homogeneous Riemannian manifolds $G/H$ such that the connected component of the identity in $H$ has an irreducible isotropy representation, 
called \emph{strongly isotropy irreducible spaces} (equivalent if $H$ is connected). 
This classification can be consulted also in \cite[Chapter~7, \S H]{Besse}, and precisely   Table~6,\,p.~203, contains our $G_2/\SO(3)$, although no more geometric or topological information is provided,
only a mention to \cite{Dyn2} to explain that the fact that $\hh$ is a maximal subalgebra of $\mathfrak{g}$ is sufficient to characterize the embedding.
Our search for another  reference of $G_2/\SO(3)$ 
has been unsuccessful.\smallskip

 From an algebraic viewpoint, 
\emph{Lie-Yamaguti algebras} (also called   \emph{generalized Lie triple systems})  are binary-ternary
algebras $(\mm,\bullet,[\ ,\ ,\ ])$ satisfying a list of 6 identities, defined precisely to translate   reductive homogeneous spaces to an algebraic structure. 
Similarly, the Lie-Yamaguti algebras which are
irreducible as modules over their Lie inner derivation algebra are the algebraic counterpart
of the isotropy irreducible homogeneous spaces. 
The work  \cite{Fabibueno} is devoted to study irreducible Lie-Yamaguti algebras of generic type.
 The generic case occurs when both the inner derivation algebra $\hh(\mm)$ and the standard enveloping algebra $\mathfrak{g}(\mm)$ are simple. Of course this is our situation. Some general facts are that $\mm$ coincides with $\hh(\mm)^\perp$, the orthogonal with respect to the Killing form of $\mathfrak{g}(\mm)$, and that $\hh(\mm)$ is a maximal subalgebra of $\mathfrak{g}(\mm)$.  Although  the pair $(G_2,A_1)$ appears explicitly in \cite[Theorem~5.1.ii)]{Fabibueno}   and in \cite[Table~9]{Fabibueno}, the work does not provide a concrete description and the classification is
transferred from the complex case.  
(Recall that we have provided in Proposition~\ref{prop_TDSppal}   a description of a principal subalgebra of $\mathfrak{g}_2$ in terms of derivations of the octonions, but not a description of its -unique- invariant complement.)   
 A concrete description (valid for the complex and the split case, but not for the compact one) appears in \cite[\S6]{Dixmier84}, where Dixmier considers simultaneously several simple nonassociative algebras defined by the
transvection of binary forms. The paper contains   some typos in the scalars involved in the $G_2$-construction, but a  corrected version appears in \cite[Theorem~4.6]{LYg2}
as follows. Denote by  $V_n$ the complex vector space  
of the homogeneous polynomials of degree $n$ in two variables $X$ and $Y$. 
For any $f\in V_n$, $g\in V_m$, consider the \emph{transvection}
\[ 
{(f,g)_q=\frac{(n-q)!}{n!}\frac{(m-q)!}{m!}
\sum_{i=0}^q(-1)^{i}\binom{q}{i}
\frac{\partial^q f}{\partial X^{q-i}\partial Y^i}\frac{\partial^q g}{\partial X^i\partial Y^{q-i}}} \in V_{m+n-2q}.
\]
The  complex Lie algebra of type $G_2$ can be obtained as   
 the standard enveloping algebra of the Lie-Yamaguti algebra
 $\mm=V_{10}$, with binary and ternary products given by 
 $$
 f_1\bullet f_2=( f_1, f_2)_5,\qquad
 [ f_1, f_2, f_3]=\frac{25}{378}(( f_1, f_2)_9, f_3)_1,
 $$
 if $f_i\in V_{10}$. Here $\hh(\mm)\cong (V_2,(\ ,\ )_1)$ turns out to be a principal three-dimensional subalgebra of $\mathfrak{g}(\mm)$.
 This Lie-Yamaguti algebra $\mm=V_{10}$ has received some attention in \cite{BH}, where its polynomial identities of low degree have been studied.
 Thus   $( V_2\oplus V_{10},[\ ,\ ])$ is a exceptional Lie algebra of type $G_2$  
 for the bracket
\begin{equation}\label{eq_transv}
 \begin{array}{l}
 {[}g_1,g_2]:=(g_1,g_2)_1,\\
 {[}g,f]:=5(g,f)_1,\\
 {[}f_1,f_2]:=  \frac{5}{378}(f_1,f_2)_9+(f_1,f_2)_5,\\
 \end{array}
\end{equation}
if $g,g_i\in V_2$ and $f,f_i\in V_{10}$.         
   The element $h=4XY\in V_2$ is ad-diagonalizable with integer eigenvalues, since
   $[h,X^kY^{2-k}]=(2-2k)X^kY^{2-k}$ and
    $[h,X^kY^{10-k}]=(10-2k)X^kY^{10-k}$. So, although the construction remains valid for the real field,  the obtained algebra is the split algebra of type $G_2$, not the compact one.
    The following result is key for our purposes. 
    
    \begin{proposition}\label{pr_conj}
    There is,  up to conjugation, only one principal three-dimensional subalgebra of the compact algebra $\mathfrak{g}_2$.
    \end{proposition}
    
    This is very well-known for the complex case, but we include one proof of the real case for completeness, due to the lack of a suitable reference. It essentially arose from conversations with A. Elduque.
    
    \begin{proof}
    We dealt with the existence in Section~\ref{se_ppal}. 
    Now assume that we have two decompositions $\mathfrak{g}_2=\mathfrak s\oplus \mathfrak m=\mathfrak s'\oplus \mathfrak m'$ with $\mathfrak s$ and $\mathfrak s'$ principal subalgebras, 
    $[\mathfrak s,\mathfrak m]\subset \mathfrak m$ and   $[\mathfrak s',\mathfrak m']\subset \mathfrak m'$, and we are going to provide an automorphism of $\mathfrak{g}_2$ which sends 
    $\mathfrak s$ to $\mathfrak s'$. 
    Denote by $\pi_{\mathfrak s}$ and $\pi_{\mathfrak m}$ the projections of $\mathfrak{g}_2=\mathfrak s\oplus \mathfrak m$ on the subspaces $\mathfrak s$ and $\mathfrak m$ 
    respectively, and similarly by
    $\pi_{\mathfrak s'}$ and $\pi_{\mathfrak m'}$ the projections with respect to the other decomposition. Write, for short, $[x,y]_{\mathfrak t}=\pi_{\mathfrak t}([x,y])$, for any $x,y\in\mathfrak{g}_2$ and any of our subspaces $\mathfrak t$ of $\mathfrak{g}_2$.
    
    First, $\mathfrak{g}_2$ is compact, so that the three-dimensional algebras $\mathfrak s$ and $\mathfrak s'$ are necessarily isomorphic to $\mathfrak{su}_2$ (there are no nilpotent elements in $\mathfrak{g}_2$). So there exists an   isomorphism of Lie algebras $\varphi\colon \mathfrak s\to \mathfrak s'$. 
    
    Second, there is only one irreducible $\mathfrak s$-module of dimension $11$. Hence, there is a bijective linear map $\rho\colon \mathfrak m\to \mathfrak m'$ such that $\rho([s,x])=[\varphi(s),\rho(x)]$ for any $s\in \mathfrak s$ and $x\in \mathfrak m$.

    Third, $\dim_\RR\hom_{\mathfrak s}(\mathfrak m\otimes \mathfrak m,\mathfrak m)=1$ since this is the situation after complexifying 
    (where the set of homomorphisms is spanned by $(\ ,\ )_5$). 
    Two $\mathfrak s$-invariant bilinear maps  from $\mathfrak m\times \mathfrak m$ to $\mathfrak m$ are
    $\rho^{-1}\circ \pi_{\mathfrak m'}\circ [\ ,\ ]\vert_{\mathfrak m'\times \mathfrak m'}\circ (\rho\times\rho)  $ and 
    $\pi_{\mathfrak m}\circ [\ ,\ ]\vert_{\mathfrak m\times \mathfrak m}$, which differ into a scalar,
    so that there is $\alpha\in\RR$ such that 
    $[\rho(x),\rho(y)]_{\mathfrak m'}=\alpha \rho([x,y]_{\mathfrak m})$ for all $x,y\in\mathfrak m$. Changing $\rho$ by $\frac{\rho}{\alpha}$, we can assume that $\alpha=1$.

    Fourth, $\dim_\RR\hom_{\mathfrak s}(\mathfrak m\otimes \mathfrak m,\mathfrak s)=1$ since this is the situation after complexifying
     (where the set of homomorphisms is spanned by $(\ ,\ )_9$).  
    Two $\mathfrak s$-invariant bilinear maps  from $\mathfrak m\times \mathfrak m$ to $\mathfrak s$ are
    $\varphi^{-1}\circ \pi_{\mathfrak s'}\circ [\ ,\ ]\vert_{\mathfrak m'\times \mathfrak m'}\circ (\rho\times\rho)  $ and 
    $\pi_{\mathfrak s}\circ [\ ,\ ]\vert_{\mathfrak m\times \mathfrak m}$, which differ into a scalar,
    so that there is $\beta\in\RR$ such that 
    $[\rho(x),\rho(y)]_{\mathfrak s'}=\beta \varphi([x,y]_{\mathfrak s})$ for all $x,y\in\mathfrak m$.
    Now define
    $$
    \Psi\colon\mathfrak{g}_2\to\mathfrak{g}_2, \ \Psi(s+x):=\varphi(s)+\rho(x),
    $$
    for any $s\in \mathfrak s$ and $x\in \mathfrak m$. Let us check that $\Psi$ is an automorphism of the algebra, with $\Psi(\mathfrak s)=\mathfrak s'$.
  From the above it is clear that, for any $s,s_1,s_2\in\mathfrak s$, $x,x_1,x_2\in\mathfrak m$,
 \begin{equation}\label{eqss}
 \begin{array}{ll}
  \Psi([s_1,s_2])=[\Psi(s_1),\Psi(s_2)],&
  \Psi([s,x])=[\Psi(s),\Psi(x)],\\
  \Psi([x_1,x_2]_{\mathfrak s})=\beta^{-1}[\Psi(x_1),\Psi(x_2)]_{\mathfrak s'},\qquad&
  \Psi([x_1,x_2]_{\mathfrak m})= [\Psi(x_1),\Psi(x_2)]_{\mathfrak m'};
  \end{array}
 \end{equation}
  so that the map $\Psi$ will be an automorphism if and only if $\beta=1$. 
  
  Fifth, use the Jacobi identity. Denote, with the indices modulo 3, the Jacobian operator  by 
  $J(x_1,x_2,x_3)=\sum_{i=1}^3[[x_i,x_{i+1}],x_{i+2}]$, which is $0$ for any choice of elements $x_i$ in $\mathfrak{g}_2$. 
  For any  $x_1,x_2,x_3\in\mathfrak m$, Eqs.~\eqref{eqss} give
   $$\begin{array}{ll}
   \pi_{\mathfrak m'}\circ\Psi([[x_1,x_2],x_3])\hspace{-2pt}&=\pi_{\mathfrak m'}\circ\Psi([[x_1,x_2]_{\mathfrak m},x_3])+\pi_{\mathfrak m'}\circ\Psi([[x_1,x_2]_{\mathfrak s},x_3])\\
 &=[[\Psi(x_1),\Psi(x_2)]_{\mathfrak m'},\Psi(x_3)]_{\mathfrak m'}+\beta^{-1}[[\Psi(x_1),\Psi(x_2)]_{\mathfrak s'},\Psi(x_3)].
   \end{array}
   $$ 
  At the same time we can apply Jacobi identity to get 
  $$\begin{array}{ll}
  0 &= \pi_{\mathfrak m'}(J(\Psi(x_1),\Psi(x_2),\Psi(x_3) )  -\Psi(J(x_1,x_2,x_3))   ) \\
  &=(\beta^{-1}-1)\sum_{i=1}^3[[\Psi(x_i),\Psi(x_{i+1})]_{\mathfrak s'},\Psi(x_{i+2})].
  \end{array}
  $$
  The expression $\sum_{i=1}^3[[\Psi(x_i),\Psi(x_{i+1})]_{\mathfrak s'},\Psi(x_{i+2})]$ cannot be identically 0 in $\mathfrak m$, since it does not vanishes after complexification. Indeed, taking in mind Eq.~\eqref{eq_transv}, chosing $\Psi(x_1)=f_1= X^{10}$, $\Psi(x_2)=f_2= Y^{10}$ and $\Psi(x_3)=f_3= XY^{9}$, the above expression coincides with
  $$
  \frac{25}{378}\sum_{i=1}^3((f_{i},f_{i+1})_9,f_{i+2})_1=\frac{25}{378}
  \big((XY,XY^9)_1+0-\frac1{10}(X^2,Y^{10})_1\big)= 
  \frac{5}{252}XY^9\ne0.
  $$
     Thus $\beta^{-1}-1=0$ and $\Psi$ is indeed an automorphism of $\mathfrak{g}_2$.
        \end{proof}   
 \medskip

We are now prepared to give a model of the homogeneous space $G_2/\SO(3)^{\textrm{irr}}$. Consider the set 
\begin{equation}\label{eq_M8}
M_8:=\{\mathfrak s\le\mathfrak{g}_2: \mathfrak s \textrm{ principal subalgebra}\}.
\end{equation}
Recall that the   group $G_2$ and the adjoint group $ \mathrm{Aut}(\mathfrak{g}_2)$ are isomorphic and a precise isomorphism is given by
$$
f\in G_2\mapsto \tilde f\in  \mathrm{Aut}(\mathfrak{g}_2),\qquad \tilde f(d)=fdf^{-1}\textrm{ for all $d\in\mathfrak{g}_2$.}
$$
 Hence $G_2$ acts on $M_8$ by $f\cdot \mathfrak s=\tilde f(\mathfrak s)=\{fsf^{-1}:s\in\mathfrak s\}$, which is a principal subalgebra of $\mathfrak{g}_2$ if $\mathfrak s$ so is.
By Proposition~\ref{pr_conj}, the action is transitive. If we denote by $H_\mathfrak s$ the isotropy group of a fixed principal subalgebra $\mathfrak s\in M_8$ (for instance, 
$\mathfrak s=\hh_8$), let us check that its Lie algebra
 $\hh_\mathfrak s$ coincides with $\mathfrak s$, which would prove that $\hh_\mathfrak s$ is principal, as required. As $H_\mathfrak s=\{f\in G_2:fsf^{-1}\in\mathfrak s \ \forall s\in \mathfrak s\}$, then its Lie algebra is
 $\hh_\mathfrak s=\{d\in \mathfrak{g}_2:[d,s]\in\mathfrak s \ \forall s\in \mathfrak s\}$, the normalizer of $\mathfrak s $, which obviously contains   $\mathfrak s $. 
 Besides there is an absolutely irreducible $\mathfrak s$-module $\mathfrak m$ such that $\mathfrak{g}_2=\mathfrak s\oplus \mathfrak m$. 
If there is $0\ne x\in \hh_\mathfrak s\setminus \mathfrak s$, we can assume without loss of generality (by subtracting its projection on $\mathfrak s$) that  
 $0\ne x\in \hh_\mathfrak s\cap\mathfrak m $.  
 On one hand, $[x,\mathfrak s]\subset \mathfrak s$ because $x$ belongs to the normalizer of $\mathfrak s$, 
 and on the other hand,  $[x,\mathfrak s]\subset [\mathfrak m,\mathfrak s]\subset\mathfrak m$. 
 Hence $[x,\mathfrak s]\subset\mathfrak s\cap \mathfrak m=0$, so that $x\in\mathfrak z(\mathfrak s)$. This contradicts the well-known fact that
 $z(\mathfrak s)=0$, which can be consulted in \cite[Chapter~6, Theorem~2.6]{enci41}.
 Alternatively, it can be directly checked in this case, simply by complexifying. 
 Indeed, an element in $\mathfrak m^\CC$ which is in the centralizer of $\mathfrak s^\CC=\textrm{span}\langle h,e,f\rangle$ would be in the one-dimensional $0$-weight space for $h$,
  but 
 the adjoint action of $e$ on the $0$-weight space is of course nonzero, giving the whole $2$-weight space and getting a contradiction.
 To summarize, we can identify $M_8$ with 
 $G_2/\SO(3)^{\textrm{irr}} $, and hence, $M_8$ described in Eq.~\eqref{eq_M8} is endowed with a manifold structure and we can think of $M_8$ as the  isotropy irreducible Wolf space.


\section{ Conclusions  }

The main purpose in this work has been to provide a complete, concrete and unified panoramic of the reductive $G_2$-homogeneous  spaces. We have tried that 
these homogeneous  spaces were explicit and the relations among them were clear, enclosing also detailed descriptions of the projections of the family of fiber bundles we can construct with the reductive $G_2$-homogeneous  spaces. We think the geometric prerequisites are modest.
The numerous references do not intend to serve as a background, but only to understand the relations among the appearing manifolds and to enrich the project, 
illustrating the interplay between Algebra and Geometry.

Some work in progress is in the following direction. Along this paper, we have focussed on the compact real form of the Lie group $G_2$. As it is well-known, there is another real form of the complex Lie algebra $\mathfrak{g}_2$, namely, the split Lie algebra $\mathfrak{g}_{2,2}$ of derivations of the split octonion algebra $\OO_s$. Its Lie group of automorphisms,   $G_{2,2}=\mathrm{Aut}(\OO_s)$,   is non-compact and not simply connected. 
Taking into account the wide generality of the results in \cite{LYg2}, it is possible to find a family of  $G_{2,2}$-manifolds closely related to the $G_2$-homogeneous manifolds $\{M_i\}_{i=1}^8$ described in this paper.
This task is not direct, for instance there will be more than $8$ reductive quotients, since $\OO_s$ possesses  different kinds of quaternion subalgebras, not all of them conjugated. This new longer family is very different from ours. The quotients are a priori non-compact, but at the same time share some remarkable properties with our compact family. In fact, the involved homogeneous spaces are reductive, and when we decompose them as a sum of irreducible modules, these decompositions do not coincide with the ones for the compact case, but their complexifications do. Thus, the dimensions of the vector spaces providing invariant metrics, invariant connections, or invariant connections with additional properties, are equal. In general, the knowledge of the specific modules appearing in the decompositions of $\mathfrak{g}_2$ as a sum of irreducible $\hh_i$-modules  for any $i=1,\dots,8$ (computed in \cite{LYg2}) provides a  very useful tool to study the corresponding homogeneous manifolds, which is not even fully exploited for all our $M_i$'s.


\end{document}